\documentclass[12pt]{amsart}
\usepackage{color}
\usepackage[all]{xy}
\usepackage{amssymb,amsmath}

\date{Feb.~8, 2010 (revision)}

\calclayout
\newtheorem{dummy}{anything}[section]
\newtheorem{theorem}[dummy]{Theorem}
\newtheorem*{thma}{Theorem A}
\newtheorem*{thmb}{Theorem B}
\newtheorem*{thmc}{Theorem C}
\newtheorem{lemma}[dummy]{Lemma}
\newtheorem{proposition}[dummy]{Proposition}
\newtheorem{corollary}[dummy]{Corollary}

\theoremstyle{definition}
\newtheorem{definition}[dummy]{Definition}
  \newtheorem{example}[dummy]{Example}
  
  \newtheorem{remark}[dummy]{Remark}
   
    \newtheorem*{question}{Question}
  \newtheorem*{acknowledgement}{Acknowledgement}

\newcommand
{\eqncount}{\setcounter{equation}{\value{dummy}}%
\addtocounter{dummy}{1}}

\newcommand{\cD}{\mathcal D}
\newcommand{\cE}{\mathcal E}
\newcommand{\cF}{\mathcal F}
\newcommand{\cG}{\mathcal G}

\newcommand{\cO}{\Or}

\newcommand{\bC}{\mathbf C}
\newcommand{\bD}{\mathbf D}

\newcommand{\bH}{\mathbf H}
\newcommand{\bZ}{\mathbf Z}

\newcommand{\bP}{\mathbf P}

\newcommand{\bbF}{\mathbb F}

\DeclareMathOperator{\Hom}{Hom}
\DeclareMathOperator{\Aut}{Aut}
\DeclareMathOperator{\Image}{im}
 \DeclareMathOperator{\Ext}{Ext}

 \DeclareMathOperator{\Ind}{Ind}
\DeclareMathOperator{\Res}{Res}
\DeclareMathOperator{\Mor}{Mor}
\DeclareMathOperator{\Map}{Map}

\DeclareMathOperator{\coker}{coker}
\DeclareMathOperator{\free}{free}
\DeclareMathOperator{\ind}{Ind}
\DeclareMathOperator{\res}{Res}
 \DeclareMathOperator{\Dim}{Dim}
\DeclareMathOperator{\hdim}{hdim}
 \DeclareMathOperator{\Syl}{Syl}

\newcommand{\cy}[1]{C_{#1}}
\newcommand{\la}{\langle}
\newcommand{\ra}{\rangle}
\newcommand{\disjointunion}{\bigsqcup}

\newcommand{\vv}{\, | \,}
\newcommand{\bd}{\partial}
\newcommand{\id}{\mathrm{id}}

\newcommand\Fp{\bbF_p}
\def\bZp{\bZ_{(p)}}
\def\Zphat{\widehat\bZ_{p}}

\def\G{\varGamma}

\def\L{\varLambda}

\DeclareMathOperator{\Is}{Iso}

\DeclareMathOperator{\Ob}{Ob}
\newcommand\Mod{\textrm{Mod}}
\newcommand{\leftexp}[2]{{\vphantom{#2}}^{ #1}{\hskip-1pt#2}}

\DeclareMathOperator{\Or}{Or}
\DeclareMathOperator{\Sub}{Sub}
\newcommand\RG{R\G}
\newcommand\ZG{\bZ\G}

\newcommand\RMod{R\textrm{-Mod}}
\newcommand\OrG{\Or _{\cF}G}
\newcommand\OrH{\Or_{\cF}H}
\newcommand\SubG{\Sub _{\cF}G}
\newcommand\SubH{\Sub_{\cF}H}
\newcommand\un{\underline}
\newcommand\uR[1]{R[{#1}\,^{\textbf{?}\,}]}
\newcommand\vR[1]{\RG(\textbf{?},{#1})}
\newcommand\uZ[1]{\mathbf{Z}[{#1}\, ^{\textbf{?}\,}]}

\newcommand\uC[1]{\CC({#1}^{\textbf{?}};R)}
\newcommand\uCZ[1]{\CC({#1}^{\textbf{?}};\bZ)}
\newcommand\CC{\bC}
\newcommand\DD{\bD}
\newcommand\PP{\bP}

\newcommand\HH{\bH}
\newcommand\RR{\un{R}}
\newcommand\uH[1]{H_*({#1}^{\textbf{?}};R)}

\newcommand\uHom{\Hom_{\RG_\textbf{?}}}
\newcommand\uExt{\Ext_{\RG_\textbf{?}}}

\newcommand\join{\divideontimes}

\begin{document}

\title{Equivariant CW-Complexes and the Orbit Category}
\author{Ian Hambleton}
\author{Semra Pamuk}
\author{Erg\"un Yal\c c\i n}

\address{Department of Mathematics, McMaster University, Canada}

\email{hambleton@mcmaster.ca }

\address{Department of Mathematics, McMaster University, Canada}

\email{pamuks@math.mcmaster.ca }

\address{Department of Mathematics, Bilkent University,
06800 Bilkent, Ankara, Turkey}

\email{yalcine@fen.bilkent.edu.tr }

\thanks{Research partially supported by NSERC Discovery Grant A4000.
The third author is partially supported by T\" UB\. ITAK-BDP and T\"
UBA-GEB\. IP/2005-16.}

\begin{abstract} We give a general framework for studying $G$-CW
complexes via the orbit category. As an application we show that the symmetric
group $G=S_5$ admits a finite $G$-CW complex $X$ homotopy equivalent
to a sphere, with cyclic isotropy subgroups.
\end{abstract}

\maketitle

\section{Introduction}
\label{sect:intro}

A good algebraic setting for studying actions of
a group $G$ with isotropy in a given family of subgroups $\cF$ is
provided by the category of $R$-modules  over the orbit category
$\G_G=\OrG$, where $R$ is a commutative ring with unit. 
 This theory was established by
Bredon \cite{bredon2},  tom Dieck \cite{tomDieck2} and L\"uck
\cite{lueck3}, and further developed by many authors (see, for
example, Jackowski-McClure-Oliver \cite[\S
5]{jackowski-mcclure-oliver2}, Brady-Leary-Nucinkis
\cite{brady-leary-nucinkis1}, Symonds \cite{symonds3},
\cite{symonds2}, Grodal \cite{grodal1}, Grodal-Smith \cite{grodal-smith1}).
In particular, the category of $\RG_G$-modules is an
abelian category with $\Hom$ and tensor product, and has enough
projectives for standard homological algebra.

\smallskip
In this paper, we will study finite group actions on spheres with
non-trivial isotropy, generalizing the approach of  Swan
\cite{swan1} to  the spherical space form problem  through periodic
projective resolutions. A finite group is said to have \emph{rank}
$k$ if $k$ is the largest integer such that $G$ has an elementary
abelian subgroup $\cy p \times \dots \times \cy p$ of rank $k$ for
some prime $p$. A rank 1 group $G$ has periodic cohomology, and Swan
showed that this was a necessary and sufficient condition for the
existence of a finite free $G$-CW complex $X$, homotopy equivalent
to a sphere.

The work of Adem-Smith \cite{adem-smith} concerning free actions on
products of spheres led to the following open problem:

\begin{question}
If $G$ is a rank 2 finite group, does there exist a finite $G$-CW
complex $X\simeq S^n$ with rank 1 isotropy~?
\end{question}

If $G$ is a finite $p$-group of rank 2, then there exist orthogonal linear
representations $V$ so that $S(V)$ has rank 1 isotropy (see
\cite{dotzel-hamrick1}). If $G$ is not of prime power order,
representation spheres with rank 1 isotropy do not exist in general:
a necessary condition is that $G$ has a $p$-effective character for
each prime $p$ dividing $|G|$ (see \cite[Thm.~47]{jackson1}). In
\cite[Prop.~48]{jackson1} it is claimed that this condition is also
sufficient for an affirmative answer to the $G$-CW question above,
but the discussion on \cite[p.~831]{jackson1} does not provide a
construction for $X$.

Our main result concerns the first non-trivial case: the permutation group
$G=S_5$ of order 120, which has rank 2 but no linear action with rank 1
isotropy on any sphere, although it does admit $p$-effective
characters for $p=2,3,5$.

\begin{thma} The permutation group $G=S_5$ admits a finite $G$-CW
complex $X \simeq S^n$, such that $X^H \neq \emptyset$ implies that
$H$ is a rank 1 subgroup of $2$-power order.
\end{thma}

\begin{remark}  
It is an interesting problem for future work to decide if the group
$G=S_5$ can act \emph{smoothly} on  $S^n$ with rank 1 isotropy.
\end{remark}

In order to prove this result we develop further techniques over the
orbit category, which may have some independent interest. A
well-known theorem of Rim \cite{rim1} shows that a module $M$ over
the group ring $\bZ G$ is projective if and only if its restriction
$\res^G_PM$ to any $p$-Sylow subgroup is projective. Over the
orbit category we have a similar statement localized at $p$ (see
Theorem \ref{thm:rim_orbit}).

\begin{thmb} Let $G$ be a finite group and let $R=\bZp$. Then an
$\RG_G$-module $M$ has a finite projective resolution with respect to
a family of $p$-subgroups if and only if its restriction $\res^G_PM$
has a finite projective resolution over any $p$-Sylow subgroup $P\leq G$.
\end{thmb}

\begin{remark} For modules over the group ring $RG$, those having
finite projective resolutions are already projective. Over the orbit
category, these two properties are distinct.
\end{remark}

Another useful feature of homological algebra over group rings is
the detection of group cohomology by restriction to the $p$-Sylow
subgroups. Here is an important concept in group cohomology (see for
example \cite{symonds1}).

\begin{definition} For a given prime $p$, we say that a subgroup
$H \subseteq G$ \emph{controls $p$-fusion} provided that
\begin{enumerate}
\item $p \nmid |G/H|$, and
\item whenever $Q \subseteq H$ is a $p$-subgroup, and there exists $g
\in G$
such that $Q^g:=g^{-1}Q g \subseteq H$, then $g = ch$ where $c \in
C_G(Q)$ and $h \in H$.
\end{enumerate}
\end{definition}

One reason for the importance of this definition is the fact that
the restriction map $$H^*(G; \Fp) \to H^*(H;\Fp)$$ is an isomorphism
if and only if $H$ controls $p$-fusion in $G$ (see \cite{mislin1},
\cite{symonds1}). We have the following generalization (see Theorem
\ref{thm:mislin_orbit}) for functors of cohomological type over the
orbit category (with respect to any family $\cF$).

\begin{thmc} Let $G$ be a finite group, $R=\bZp$, and $H\leq G$ a
subgroup
which controls $p$-fusion in $G$. If $M$ is an $\RG_G$-module and
$N$ is a cohomological Mackey functor, then the restriction map
$$\res^G_H\colon \Ext^n_{\RG_G}(M,N) \to  \Ext^n_{\RG_H}(\res^G_HM,
\res^G_HN)$$
is an isomorphism for $n\geqslant 0$, provided that the centralizer
$C_G(Q)$ of any $p$-subgroup $Q \leq H$, with $Q \in \cF$, acts trivially on $M(Q)$ and
$N(Q)$.
\end{thmc}

The construction of the $G$-CW complex $X$ for $G=S_5$ and the proof of Theorem A is carried
out in Section \ref{sect:construction-new}.  We first construct finite projective chain complexes
$\CC^{(p)}$ over the orbit categories $\RG_G$, with $R = \bZp$,  separately for the prime $p = 2, 3, 5$ dividing $|G|$. In each case, the isotropy family  $\cF$ consists of
the rank 1 subgroups of $2$-power order in $G$. 

The chain complexes $\CC ^{(p)}$ all have the same \emph{dimension function} (see Definition \ref{def: monotone}).   We prescribe a non-negative function
$\un{n}\colon \cF \to \bZ$, with the property that $\un{n}(K) \leq \un{n}(H)$ whenever $H$ is conjugate to a subgroup of $K$. Then, by construction, 
 each complex $\CC ^{(p)}$ has the $R$-homology of an $\un{n}$-sphere: 
for each $K\in \cF$, the
complexes $\CC^{(p)}(K)$ have homology  $H_i =R$ only in two
dimensions $i=0$ and $i= \un{n}(K)$. In other words, the complexes $\CC ^{(p)}$ are algebraic versions of tom Dieck's \emph{homotopy representations} \cite[II.10]{tomDieck2}.

In the case $p=2$, we start with the group $H = S_4$ acting by orthogonal rotations on
 the $2$-sphere. A regular $H$-equivariant triangulation of an inscribed cube or octahedron gives a finite projective chain complex over $\RG _H$. 
 Then we use Proposition \ref{prop:mislin_chain}, a chain complex version of Theorem C, to lift
it to a finite projective complex over $\RG _G$. 
For $p=3$ and
$p=5$, the $p$-rank of $S_5$ is $1$, and there exists a periodic
complex over the group ring $RG$ (see Swan \cite[Theorem B]{swan1}). 
We start with a periodic complex
over $RG$ and add chain complexes to this complex, for every
nontrivial subgroup $K \in \cF$, to obtain the required complex $\CC^{(p)}$ over $\RG_G$. 

 We use the theory of algebraic Postnikov
sections by Dold \cite{dold1} to glue the complexes together to form
a finite projective $\ZG_G$ chain complex (see Section \ref{sect:chain complexes}). We complete the chain
complex construction by varying the finiteness obstruction  to obtain
a complex of free $\ZG_G$-modules, and then we prove a realization
theorem (see Section \ref{sect:realization}) to construct the
required $G$-CW complex $X \simeq S^n$.

Throughout the paper, a \emph{family} of subgroups will always mean
a collection of subgroups which is closed under conjugation and
taking subgroups. Also, unless otherwise  stated, all  modules are
\emph{finitely generated}.

\begin{acknowledgement} The authors would like to thank the referee for many valuable criticisms and suggestions. The third author would also like to thank McMaster University for the support provided by a H.~L.~Hooker Visiting Fellowship, and the Department of Mathematics \& Statistics at McMaster for its hospitality while this paper was written. 
\end{acknowledgement}
\section{Modules over small categories }
\label{sect:definitions}

Our main source for the material in this section is L\"uck 
 \cite[\S 9, \S 17]{lueck3} (see also \cite[\S I.10, \S I.11]{tomDieck2}). We include it here for the
convenience of the reader.

Let $R$ be a commutative ring.  We denote the category of
$R$-modules by $\RMod$. For a small category $\G$ (i.e., the objects
$\Ob(\G)$ of $\G$ form a set), the category of right $\RG$-modules is defined
as the category of contravariant functors $\G \rightarrow  \RMod$,
where the objects are functors $M(-)\colon \G\rightarrow R$-$\Mod$
and morphisms are natural transformations. Similarly, we define the
category of left $\RG$-modules as the category of covariant functors
$N(-)\colon \G\rightarrow  R$-$\Mod$. We denote the category of
right $\RG$-modules by Mod-$\RG$ and the category of left
$\RG$-modules by $\RG$-Mod.

The category of covariant or contravariant functors from a small
category to an abelian category has the structure of abelian
category which is object-wise induced from the abelian category
structure on abelian groups (see \cite[Chapter 9, Prop.
3.1]{maclane1}).  Hence the category of $\RG$-modules is an abelian
category where the notions submodule, quotient module, kernel,
image, and cokernel are defined object-wise.  The direct sum of
$\RG$-modules is given by taking the usual direct sum object-wise.

\begin{example} The most important example for our applications is
the orbit category of a finite group. Let $G$ be a finite group and
let $\cF$ be a family of subgroups of $G$ which is closed under
conjugation and taking subgroups. The orbit category $\cO (G)$ is
the category whose objects are subgroups $H$ of $G$ or coset spaces
$G/H$ of $G$, and the morphisms $\Mor (G/H,G/K)$ are given by the
set of $G$-maps $f\colon G/H\to G/K$. 

The category $\G_G=\cO _{\cF }G$
is defined as the full subcategory of $\cO (G)$ where the objects
satisfy $H\in\cF $. The category of right $\RG_G$-modules is the
category of contravariant functors from $\OrG$ to $R$-modules. A
right $\RG_G$-module $M$ is often called a \emph{coefficient system}
\cite{symonds2}. We will sometimes denote $M(G/H)$ by $M(H)$ if the group
$G$ is clear from the context. When $\cF =\{e\}$, $\RG_G$-$\Mod$ is
just the category of left $RG$-modules and $\Mod$-$\RG_G$ is just the
category of right $RG$-modules. \qed
\end{example}

Now, we will introduce the tensor product and $\Hom$ functors for
modules over small categories. Let $\G$ be a small category and let
$M\in\Mod$-$\RG$ and $N\in \RG$-$\Mod$.  The tensor product over $R
\G$ is given by
$$M\otimes_{\RG} N=\bigoplus_{x \in \Ob(\G) }M(x)\otimes N(x)/\sim $$ where
$\sim$ is the equivalence relation defined by $\varphi ^*(m)\otimes
n \sim m\otimes\varphi _*(n)$ for every morphism $\varphi : x \to
y$. For $\RG$-modules $M$ and $N$, we mean by $\Hom_{\RG}(M,N)$ the
$R$-module of $\RG$- homomorphisms from $M$ to $N$. In other words,
$$\Hom_{\RG}(M,N)\subseteq \bigoplus_{x \in \Ob(\G) }\Hom_{R}(M(x),N(x))$$
is the submodule satisfying the relations $f(x)\circ \varphi^* =
\varphi^*\circ f(y)$, for every morphism $\varphi\colon x \to y$. We
sometimes consider a second tensor product, namely the tensor
product over $R$, which is defined for $\RG$-modules $M$ and $N$
which are both left modules or both right modules. The tensor
product $M\otimes_R N$ is defined by the formula
$$[M \otimes _R N] (x) =M(x)\otimes _R N(x)$$ on objects $x \in \Ob (\G)$
and on morphisms, one has $[M \otimes _R N ] (f)= M(f) \otimes _R
N(f)$.

The tensor product over $\RG$ and $\Hom _{\RG}$ are adjoint to each
other. This can be described in the following way:

\begin{proposition}
\label{prop:adjoint} Given two small categories $\G$ and $\L$, the
category of $\RG$-$R\L$-bimodules is defined as the category of
functors $\G \times \L ^{{\rm op}} \to \RMod$. For a right
$\RG$-module $M$, an $\RG$-$R\L$-bimodule $B$, and a right
$R\L$-module $N$, one has a natural transformation $$ \Hom _{R\L} (M
\otimes _{\RG} B, N ) \cong \Hom _{\RG } (M, \Hom _{R\L } (B, N
)).$$
\end{proposition}

\begin{proof}
See \cite[9.21.]{lueck3}
\end{proof}

We will be using this isomorphism later when we are discussing
induction and restriction.

\subsection{Free and finitely generated modules} 
For a small category $\G$, a sequence $$ M'\rightarrow  M
\rightarrow  M''$$ of $\RG$-modules is exact if and only if
$$ M'(x)\rightarrow M(x)\rightarrow M''(x)$$ is exact for
all $x\in \Ob( \G )$. Recall that a module $P$ in $\Mod$-$\RG$ is
projective if the functor $$ \Hom_{\RG}(P,-)\colon \Mod\textrm{-}\RG
\rightarrow \RMod$$ is exact. For an object $x \in \G$, we
define a right $\RG$-module $\vR{x}$ by setting $$\vR{x}(y) = R\Mor
(y,x)$$ for all $y \in \Ob (\G )$. Here, $R\Mor (y,x)$ denotes the
free abelian group on the set of morphisms $\Mor (y,x)$ from $y$ to
$x$. As a consequence of the Yoneda lemma, we have $$ \Hom_{\RG}(
\vR{x}, M) \cong M(x).$$ So, for each $x \in \Ob (\G )$, the module
$\vR{x}$ is a projective module. When working with modules over
small categories one uses the following notion of free modules.

\begin{definition} Let $\G$ be a small category. A $\Ob (\G)$-set is
defined
as a set $S$ together with a map $\beta \colon S \to \Ob (\G)$. We say a
$\RG$-module $M$ is free if it is isomorphic to a module of the form
$$\RG (S)= \bigoplus _{b \in S} \vR{\beta (b) }$$
for some $\Ob (\G)$-set $S$. A free module $\RG (S)$ is called
finitely generated if the set $S$ is finite.
\end{definition}

Note that for every $\RG$-module $M$, there is a free $\RG$-module
$\RG(S)$ and a map $f\colon \RG (S) \to M$ such that $f$ is
surjective. We can take such a free module by choosing a set of
generators $S_x$ for the $R$-module $M(x)$ for each $x\in \Ob (\G)$,
and then taking $S$ as the $\Ob (\G)$-set which has the property
$\beta ^{-1} (x)=S_x$. A free module $\RG (S)$ which maps
surjectively on $M$ is called a \emph{free cover} of $M$. A
$\RG$-module is called \emph{finitely generated} if it has a
finitely generated free cover. 

It is clear from our description
of free modules  that an $\RG$-module $M$ is projective
if and only if it is a direct summand of a free module. This shows
that the module category of a small category has enough projectives.
We will later give a more detailed description of projective
$\RG$-modules.

\begin{example}\label{ex: basic-modules}
 For the orbit category $\G=\Or(G)$, the free modules described above
have a more specific meaning. For any subgroup $K\leq G$,
 the $\RG$-module $\vR{G/K}$ is given by  
 $$\vR{G/K}(G/H) = R\Mor(G/H,G/K)=R[(G/K )^H]$$
  where $R[(G/K)^H]$ is the free abelian group
on the set of fixed points of the $H$ action on $G/K$. Because of
this, we denote the free module $\vR{G/K}$ by $\uR{G/K}$. 

If $\cF$ is a family of subgroups, and $\G_G = \OrG$, we obtain free $\RG_G$-modules
$\uR{G/K}$ by restriction whenever $K \in \cF$. 
The \emph{constant} $\RG_G$-module $\un{R}$ defined by $\un{R}
(H)=R$, for all $H\in \cF$, is just the restriction to $\RG_G$ of the module $\un{R}=\uR{G/G}$.
This shows that the constant module $\un{R}$ is projective if $G\in
\cF$.
More generally, if $K \in \cF$, a non-empty  fixed set 
$$(G/K)^H = \{ gK \vv g^{-1}Hg \subseteq K\} \neq \emptyset$$
 implies $H \in \cF$, since $\cF$ is closed under conjugation and taking subgroups. Therefore, $\uR{G/K}(H) = 0$ for $H \notin \cF$, whenever $K \in \cF$.
 
\end{example}

\subsection{Induction and Restriction} We now recall the definitions and terminology  for these terms presented in L\"uck \cite[9.15]{lueck3}.
Let $\G$ and $\L$ be two small categories. Given a covariant functor
$F\colon \L \rightarrow \G$, we define an $R\L$-$\RG$-bimodule
$$R(\mathbf{??}, F(\mathbf{?}))\colon \L\times \G ^{{\rm op}} \to
\RMod$$ on objects by $(x,y) \to R\Hom (y, F(x))$. We define the
restriction map
$$\Res_F\colon \Mod\textrm{-}\RG \rightarrow  \Mod\textrm{-}R \L $$
as the composition with $F$. The induction map $$\Ind _F\colon
\Mod\textrm{-} R \L \rightarrow  \Mod\textrm{-}\RG$$ is defined by
$$\Ind _F(M)(??)=M\otimes_{R\L}R(\mathbf{??},F(\mathbf{?})) $$
for every $R\L$-module $M$. For every right $\RG$-module $N$, the
$R\L$-module $$\Hom_{\RG}( R(\mathbf{??},F(\mathbf{?})), N)$$ is the
same as the composition $\L \xrightarrow{F} \G \xrightarrow{N}
\RMod$.  So, by Proposition \ref{prop:adjoint}, we can conclude the
following:

\begin{proposition}\label{prop: induction-restriction}
Induction and restriction are adjoint functors: for any $\RG$-module
$M$ and $R \L$-module $N$, there is a natural isomorphism
$$\Hom_{\RG} (\Ind _F M,N)=\Hom_{R\L}(M,\Res_F N).$$
The induction functor respects direct sum, finitely generated, free,
and projective but it is not exact in general. The restriction
functor is exact but does not respect finitely generated, free, or
projective in general.
\end{proposition}

Now we will define functors which are special cases of the restriction
and induction functors. Let $\G$ be a small category. For $x\in
\Ob(\G)$, we define $R[x]=R\Aut(x)$ to be the group ring of the
automorphism group $\Aut(x)$ and denote the category of right
$R[x]$-modules by $\Mod$-$R[x]$. Let $\G _x$ denote the full
subcategory of $\G$ with single object $x$ and let $F\colon \G _x
\to
\G$ be the inclusion natural transformation. The restriction functor
associated to $F$ gives a functor
$$\Res _x\colon  \Mod\textrm{-}\RG \to \Mod\textrm{-}R[x]$$
which is called the {\bf restriction functor}. This functor behaves
like an evaluation map $\Res _x (M) =M(x)$. In the other direction, 
the induction functor associated to $F$ gives a functor
$$E_x\colon \Mod\textrm{-}R[x]\rightarrow
\Mod\textrm{-}\RG$$ which is called the {\bf extension functor}. For
a $R[x]$-module $M$, we  define $E_x(M)(y)=M\otimes_{R[x]}R\Mor
(y,x)$ for every $y\in \Ob(\G)$. They form an adjoint pair: for every
$R[x]$-module $M$ and an $\RG$-module $N$, we have
$$ \Hom _{\RG} (E_x M , N) \cong \Hom _{R[x]} (M, \Res_x N).$$

By general properties of restriction and induction, the functor
$\Res _x$ is exact and $E_x$ takes projectives to projectives. In
general, $E_x$ is not exact and $\Res _x$ does not take projectives
to projectives. But in some special cases, we can say more. For
example, when $\G$ is free, i.e.~$R\Mor (y,x)$ is a free
$R[x]$-module for all $y\in\G$, then it is easy to see that $E_x$ is
exact \cite[16.9]{lueck3}.

\begin{example} In the case of an orbit category $\G_G=\OrG$, we denote
the extension function for $H \in \cF $ simply by $E_H$ and the
restriction functor by $\Res _H$.  In this case, the automorphism
group $\Aut (G/H)$ for $H \in \cF$ is isomorphic to the quotient
group $N_G (H)/H$. The isomorphism $ N_G (H)/H \cong \Aut (G/H)$ is
given by the isomorphism $nH\to f_n$ where $f_n (gH)=gn^{-1}H$ for
$n \in N_G (H)$ (see \cite[Example 11.2]{tomDieck2}). This
isomorphism takes right $R[x]$-modules to right $R[N_G
(H)/H]$-modules, so given a right $\RG$-module $M$, the evaluation
at $H \in \cF$ gives a right $R[N_G (H)/H]$-module. 

It is easy to see
that the morphism set $\Mor (G/K, G/H)$ is a free $[N_G(H)/H]$-set,
so $\OrG$ is free in the above sense \cite[Example 16.2]{lueck3}.
Therefore, the functor $E_H$ is exact and preserves projectives,
whereas $\Res _H$ is exact but does not necessarily preserve
projectives. For example, the module $\uZ{G/G}$ is free over $\bZ\cO (G)$ by definition, but
$\Res_H\uZ{G/G}=\bZ$ is not projective whenever $N_G(H)/H\neq 1$.
\end{example}

\subsection{Inclusion and Splitting Functors}
We will introduce two more functors. These are also
special cases of induction and restriction, but they are defined
through a bimodule rather than just a natural transformation $F$. We
first describe these functors and then give their interpretations as
restriction and induction functors.

Let $\G$ be an EI-category. By this, we mean that $\G$ is a small
category where every endomorphism $x\to x$ is an isomorphism for all
$x\in \Ob (\G)$. This allows us to define a partial ordering on the
set $\Is(\G)$ of isomorphism classes $\bar x$ of objects $x$ in $\G$. 
For $x,y \in \Ob (\G)$, we say $\bar x\leq \bar y$ if
and only if $\Mor(x,y)\neq\emptyset$. The EI-property ensures that $\bar x\leq \bar y\leq \bar x$ implies $\bar x = \bar y$.

For each object $x\in \G$, and $M \in  \Mod\textrm{-}
R[x]$, the
{\bf inclusion functor}, $$I_x\colon \Mod\textrm{-}
R[x]\rightarrow  \Mod\textrm{-}\RG$$ is defined by
$$I_xM(y)={\begin{cases}
M\otimes_{R[x]}R\Mor(y,x) & \, \textrm{if} \quad \bar{y}=\bar{x} \\
\{0\} &\,  \textrm{if} \quad \bar{y}\neq \bar{x}.
\end{cases} }$$
In the other direction, we define the {\bf splitting functor}
$$ S_x\colon \Mod\textrm{-}\RG\rightarrow  \Mod\textrm{-}R[x]$$
by $S_x(M)=M(x)/M(x)_s$ where $M(x)_s$ is the $R$-submodule of
$M(x)$ which is generated by the images of $M(f)\colon
M(y)\rightarrow M(x)$ for all $f\colon x\rightarrow  y$ with $\bar x \leq
\bar y$ and $\bar x \neq \bar y$.

There is a $\RG$-$R[x]$-bimodule $B$ defined in such a way that
the inclusion functor $I_x$ can be described as $M\to \Hom _{R[x]}
(B, M)$ and the splitting functor $S_x$ is the same as the functor
$M \to M \otimes _{\RG } B $ (see \cite[page 171]{lueck3} for
details). So $(S_x, I_x)$ is an adjoint pair, meaning that
$$\Hom_{R[x]}(S_xM,N)\cong\Hom_{\RG}(M,I_xN)$$
for every $\RG$-module $M$ and $R[x]$-module $N$.

 From general properties of induction and restriction, we can
conclude that $I_x$ is exact and $S_x$ preserves projectives. Some
of the other properties of these functors are listed in \cite[Lemma
9.31]{lueck3}. It is interesting to note that the composition $S_x
\circ E_x$ is naturally equivalent to the identity functor. Also,
the composition $S_y \circ E_x$ is zero when $\bar{x} \neq \bar{y}$.
These are used to give a splitting for projective $\RG$-modules.

\begin{theorem}
Let $P$ be a finitely generated projective $\RG$-module. Then
$$P\cong\bigoplus_{x\in \Is(\G)} E_xS_x(P).$$
\end{theorem}

\begin{proof}
For proof see \cite[Corollary 9.40]{lueck3}.
\end{proof}
In the statement,  the notation $\bigoplus_{x\in \Is(\G)}$ means that the
sum is over a set of representatives  $x\in \Ob(\G)$ for $\bar x \in \Is(\G)$.

\subsection{Resolutions for $\RG$-modules}\label{subsection:resolutions}
Let $\G$ be an EI-category.  For a non-negative integer $l$ we define an
$l$-chain $c$ from $x\in \Ob(\G)$ to $y\in \Ob(\G)$ to be a sequence
$$ c: \ \bar x=\bar x_0<\bar x_1<\cdots <\bar x_l=\bar y\ .$$
Define the length $l(y)$ of
$y\in \Ob(\G)$ to be the largest integer $l$ such that there exists
an $l$-chain  from some $x \in \Ob (x)$ to $y$. The
\emph{length} $l(\G)$ of $\G$ is $\textrm{max}\{l(x)\vv x\in
\Ob(\G)\}$.  Given an $\RG$-module $M$, its length $l(M)$ is defined
by $\textrm{max}\{l(x)\vv M(x) \neq 0 \}$ if $M$ is not the zero
module and $l(\{0\})=-1$.

We call a category $\G$ \emph{finite} if $\Is(\G)$ and $\Mor(x,y)$
are finite for all $x,y\in \Ob(\G)$. Denote by $m(\G)$ the least
common multiple of all the integers $|\Aut (x)|$.

Given an $\RG$-module $M$, consider the map
$$\phi\colon \bigoplus_{x\in \Is(\G)}E_x\Res_xM \to M$$
where for each $x\in \Ob(\G)$, the map $\phi_x\colon E_x \Res _x M
\to M$ is the map adjoint to the identity homomorphism. It is easy
to see that $\phi$ is surjective. Let
$$EM:=\bigoplus_{x\in
\Is(\G)}E_x\Res_xM$$
and let $KM$ denote the kernel of $\phi\colon
EM\to M$. Note that if $x$ is an object with $l(x)=l(M)$, then $\Res
_x=S_x$ which also gives that $$\Res _x \phi\colon \Res _x E_x \Res
_x M\to \Res _x M$$ is an isomorphism. Note that this implies $l(KM)
< l(M)$ which allows one to proceed by induction and obtain the
following:

\begin{proposition}\label{prop:resolution}
If $\G$ is finite EI-category, then every nonzero $M$ has a finite
resolution of the form $$0\rightarrow  EK^tM\rightarrow
\cdots\rightarrow EKM\rightarrow  EM\rightarrow  M\rightarrow  0\
.$$ where $t=l(M)$.
\end{proposition}

\begin{proof} See {\cite[17.13 ]{lueck3}}. Here $K^0M = M$ and $K^sM = K(K^{s-1}M)$.
\end{proof}

From the description above it is easy to see that
$$EK^sM:=\bigoplus_{x\in \Is(\G)}E_x \res _x K^sM$$ where $\res _x
K^s M$ is isomorphic to a direct sum  of $R[x]$-modules
$$M(c):= M(x_0)\otimes_{R[x_0]}R\Mor (x_1,x_0)\otimes_{R[x_1]}\cdots
\otimes_{R[x_{s-1}]}R\Mor (x,x_{s-1})$$ over
 representatives in $\Ob(\G)$ for all the chains of the
form $c: \bar x<\bar x_{s-1}<\cdots<\bar x_0$  (see \cite[17.24]{lueck3}). Note that
if $\G$ is a finite, free EI-category, then the resolution given in
Proposition \ref{prop:resolution} will be a finite projective
resolution if $M(c)$ is projective as an $R[x]$-module for every
chain $c$. This gives the following:

\begin{proposition} Let $M$ be $\RG_G$ module where $\G_G=\OrG$ for some
finite group $G$ and $R$ is a commutative ring such that $|G|$ is
invertible in $R$. Suppose also that $M(H)$ is projective as an
$R$-module for all $H \in \cF$. Then, $M$ has a projective
resolution with length less than or equal to $l(\G)$.
\end{proposition}

\begin{proof} See \cite[17.31]{lueck3}.
\end{proof}

In particular, if $R=\bZp$ with $p \nmid |G|$ and if $M$ is a
$\RG$-module such that $M(H)$ is $R$-torsion free for all $H \in \cF$,
then $M$ has a finite projective resolution of length less than or
equal to $l(M)$.


\section{The proof of Theorem B}
\label{sect:Rim's thm}

The main result of this section is Theorem \ref{thm:rim_orbit}, which is an orbit category version of a well-known result of Rim \cite{rim1}. We first collect some further information about induction and restriction for subgroups.

Let $G$ be a finite group and let $H$ be a subgroup of $G$. Given a
family of subgroups $\cF$ of $G$, we  consider the orbit
categories $\G_G=\OrG$ and $\G_H =\OrH$, where the objects of $\G_H$
are orbits of $H$ with isotropy in $\cF_H=\{K \leq H\vv K \in
\cF\}$. Let $F \colon \G_H \to \G_G$ be the functor which takes
$H/K$ to $G/K$ and sends an $H$-map $f\colon H/K \to H/L$ to the
induced $G$-map
$$\ind _H ^G (f) \colon G/K = G \times_H H/K \to G \times_H H/L =G/L$$
for every $K, L\in \cF_ H$. Note that if $f$ is the map which takes
$eK$ to $hL$, then $\ind _H ^G (f) (gK)=ghL$. The restriction and
induction functors (see Proposition \ref{prop: induction-restriction}) associated to this functor gives us two adjoint
functors
$$ \res ^G _H \colon \Mod\text{-}\G _G \to \Mod\text{-}\G _H$$
and
$$\ind _H ^G \colon \Mod\text{-}\G _H \to \Mod\text{-}\G _G .$$
The restriction functor is  defined  as the composition
with $F$. So, for a $\RG _G$-module  $M$, we have $(\res ^G _H
M)(K)=M(K)$, for all $K \in \cF_H$. For the induced module we have the
following formula:

\begin{lemma}\label{lem:induction} Let $N$ be a $\RG_H$-module
and $K \leq G$. Then,
\begin{equation*}
(\ind ^G _H N) (K)\cong\bigoplus_{gH\in G/H,\ K^g\leq H} N(K^g)
\end{equation*}
where $K^g=g^{-1}Kg$.
\end{lemma}

\begin{proof}
For a (right) $\RG _H$-module $N$, the induced module $\ind _H ^G N$
is defined as the direct sum
$$\bigoplus _{L\leq H} N (L) \otimes _R
R\Mor (G/K, G/L)$$ modulo the relations $n \otimes \varphi f \sim
\varphi ^* (n) \otimes f$ where $n \in N(L)$, $f \in \Mor (G/K,
G/L')$ and $\varphi =\ind _H ^G (\phi)$ for some $\phi : H/L'\to
H/L$. Every morphism $G/K \to G/L$ which satisfies the condition $L
\leq H$ can be written as a composition $\varphi f_g$ where $\varphi
\colon G/K ^g \to G/L$ is induced from an $H$-map and $f _g \colon
G/ K \to G/K^g $ is given by $xK\to xgK^g$, for some $g \in G$. 

This
shows that every element in the above sum is equivalent to an
element of the form $n \otimes f_g$ where $n \in N(K^g)$ and $f _g
\colon G/ K \to G/ K^g $ is as above with $K^g \leq H$. There is one
summand for each $gH$ satisfying $K^g \leq H$.
\end{proof}

Note that we can also express the above formula by
$$(\ind ^G _H N) (K)\cong\bigoplus_{gH\in (G/H)^K} N(K^g).$$ If
$J\leq K$, then the argument above can be extended to show that
restriction map
$$(\ind ^G _H N) (K)\to (\ind ^G _H N) (J)$$ is given by the
coordinate-wise restriction maps $N(K^g ) \to N(J^g)$. Note that if
$gH \in (G/H )^K$, then $gH \in (G/H )^J$. Similarly, the conjugation
map $$(\ind^G _H N) (K)\to (\ind ^G _H N) (\leftexp{x} K)$$ can be
described by coordinate-wise conjugation maps. From these, it is
easy to see that $\ind _H ^G \underline{R} \cong \uR{G/H}$. A
generalization of this argument gives the following:

\begin{lemma}\cite[Cor.~2.12]{symonds2}.\label{lem:ind-res formulas1}
Let $G$ be a finite group and let $H$ be a subgroup of $G$. For
every $\RG_G$-module $M$, we have $\ind_H^G\res_H^G M\cong M
\otimes_R \uR{G/H}$.
\end{lemma}

We also have the following formulas:

\begin{lemma}\addtolength{\itemsep}{0.2\baselineskip}
\label{lem:ind-res formulas2} Let $G$ be a finite group and let $H$
be a subgroup of $G$.
\begin{enumerate}
\item For every $ K \leq H$, we have $\ind _H ^G
\uR{H/K}\cong\uR{G/K}$.
\item For every $K\leq G$, we have $\res ^G _H \uR{G/K}\cong \bigoplus
_{K\backslash G/H} \uR{H/(H\cap\leftexp{g}K)}$.
\end{enumerate}
\end{lemma}

\begin{proof}
Part (i) follows from the fact that
$\ind _H ^G \ind _K ^H = \ind _K ^G $ which is a consequence of a
more general formula $\ind _F \ind _{F'}=\ind _{F\circ F'}$. We can
prove this more general formula by using adjointness and the formula
$\res _{F'} \res _{F}=\res _{F \circ F'}$. For (ii), observe that
the definition of $\uR{G/H}$ can be extended to define a $\RG
_G$-module $\uR{S}$  for every $G$-set $S$,  by taking
\eqncount
\begin{equation}\label{basic G-set modules}
\uR{S}(G/K) = R\Map _G(G/K ,S)
\end{equation}
 for every $K \in \cF$, where
$\Map _G(G/K ,S)$ denotes the set of $G$-maps from $G/K$ to $S$. For
$G$-sets $S$ and $T$, we have an isomorphism $\uR{(S\disjointunion T)}\cong
\uR{S} \oplus \uR{T}$. By the definition of restriction map, we get
$$\Bigl (\res _H ^G \uR{S}\Bigr )(H/K)= R\Map _G (G/K, S)=
R\Map _H (H/K, \res ^G _H S).$$ It is easy to see that this induces
an $\RG_H$-module isomorphism
$$\res _H ^G \uR{S}\cong \uR{(\res^G _H S)}.$$ Since
$$ \res ^G _H (G/K)\cong \coprod _{H\backslash G /K } H/ (H \cap
\leftexp{g} K)$$ as $G$-sets, we obtain the formula given in (ii).
\end{proof}

\begin{example}\label{exmp:restrictionformula}
Let $G=S_5$ be the symmetric group on $\{1,2,3,4,5\}$
and $H=S_4$ be the subgroup of symmetries that fix $5$. Let $C_2
=\la (12) \ra$ and $C_3=\la (345) \ra$. The formula in Lemma
\ref{lem:ind-res formulas2} (ii) gives
$$ \res ^G _H \uR{G/(C_2\times C_3)} =\uR{H/C_2} \oplus \uR{H/\leftexp{g}{C_3}}$$
where $\leftexp{g}{C_3} = \la (123)\ra$.
 From this expression we obtain
$$\uR{G/(C_2\times C_3)}(G/C_2)\cong \uR{H/C_2} (H/C_2) \cong R[N_H (C_2)/C_2 ],$$
 as an $N_H (C_2)/C_2$-module,
where $N_H (C_2) = C_2\times C_2$. Note that $N_G (C_2)=C_2 \times
S_3$ and as an $N_G (C_2)/C_2$-module $\uR{G/(C_2\times
C_3)}(G/C_2)$ is isomorphic to $R[C_2 \times S_3 / C_2 \times C_3]$.  \qed
\end{example}

We can give a more general formula for $\uR{G/H} (G/K)$ as follows:

\begin{lemma}\label{lem:fixedpointformula}
Let $G$ be a finite group, and $H$ and $K$ be two subgroups of $G$.
Then, as an $R[N_G (H)/H]$-module $$\uR{G/K}(G/H) \cong \bigoplus
_{v(H, K)} R\big [N_G (H) / N_{\leftexp{g}{K}} (H ) \big ]$$ where the sum is
over the set $v(H,K)$ of representatives of $K$-conjugacy classes  of
subgroups $H^g$ such that  $H^{g} \leq K$.
\end{lemma}

\begin{proof}
This formula can  easily be proved by first determining the orbits of
$N_G (H)$ action on $(G/K) ^H = \{ gK \vv H^g \leq K\}$, and then by calculating the isotropy
subgroups for each of these orbits. A similar computation can be
found in the proof of Theorem 4.1 in \cite{coskun-yalcin1}.
\end{proof}

\begin{proposition}
Both $\res ^G _H$ and $\ind _H ^G$ are exact and take projectives to
projectives.
\end{proposition}

\begin{proof}
The fact that $\res ^G _H$ is exact and $\ind _H ^G$ preserves
projectives follows from the general properties of restriction and
induction functor associated to a natural transformation $F$. The
fact that $\ind^G _H$ is exact follows from the formula given in
Lemma \ref{lem:induction}. Finally, to show that $\res^G _H$ takes
projective to projectives, it is enough to show it takes free
modules to projective modules. An indecomposable free $\RG
_G$-module $M$ is of the form $\uR{G/K}$ for some $K \in \cF$. By
Lemma \ref{lem:ind-res formulas2}, $\res ^G _H (\uR{G/K} )$ will be
projective if $H \cap \leftexp{g} K$ is in $\cF$ for all $HgK \in H
\backslash G /K$. But this is always true since the family $\cF$ is
closed under conjugation and taking subgroups.
\end{proof}

A result of Rim \cite{rim1} relates projectivity over the group ring
$\bZ G$ to projectivity over the $p$-Sylow subgroups.

\begin{proposition}[Rim's Theorem] Let $G$ be a finite group, and $M$
be a finitely generated $\bZ G$-module. Then $M$ is projective over
$\bZ G$ if and only if $\res^G_P M$ is projective over $\bZ P$ for
any $p$-Sylow subgroup $P \leq G$.
\end{proposition}

\begin{proof}
A module $M$ is $\bZ G$-projective if and only if $\Ext^1_{\bZ
G}(M,N)=0$ for every $\bZ G$-module $N$. Therefore $M$ is projective
if and only if $\bZp \otimes _{\bZ} M$ is projective over $\bZp G$
for all primes $p$ dividing the order of $G$.

For any $p$-Sylow subgroup $P \leq G$,  the permutation module
$R[G/P] \cong R \oplus N$ splits when $R = \bZp$. Therefore, if $M$ is any
$RG$-module, $M\otimes_R R[G/P] \cong M \oplus (M\otimes_R N)$.
Since $M\otimes_R R[G/P] \cong \ind_P^G\res^G_P M$, the projectivity
of $M$ is equivalent to the projectivity of $\res^G_P M$.
\end{proof}

Here is an orbit category version of this result. 

\begin{theorem}[Rim's Theorem for the Orbit Category]
\label{thm:rim_orbit} Let $G$ be a finite group and let $M$ be a $R
\G_G$-module where $R=\bZp$. Suppose that $\cF $ is a family of
$p$-subgroups in $G$. Then $M$ has a finite projective resolution
 if and only if $\res^G_PM$ has a finite projective resolution for any
$p$-Sylow subgroup $P$ of $G$.
\end{theorem}

\begin{proof}  One direction is clear since $\res^G_P$ is
exact and takes projectives to projectives.  For the other
direction, we will give the proof by induction on the length $l(M)$
of $M$.  Without loss of generality, we can assume that $M(H)$ is $R$-torsion free for all $H\in \cF$. Suppose $M$ is a $\RG_G$-module with $l(M)=0$.  Then, we can
regard $M$ as an $RG$-module.  If $\res^G_P M$ has a finite
projective resolution, then $\res^G _PM$ must be projective (see
\cite[page 348]{lueck3}). Then, by Rim's theorem, $M$ is a
projective $RG$-module, hence has finite projective length.

Now, assume $M$ is an $\RG _G$-module with $l(M)=s>0$.  Let
$$ 0\rightarrow  P_n\rightarrow \cdots\rightarrow  P_0
\rightarrow \res^G_PM\rightarrow  0$$ be a projective resolution for
$\res^G_PM$. We can assume that $l(P_i) \leqslant s$ for all $i$. Then,
for $Q \in \cF$ with $l(Q)=s$, we have
$$S_Q P_i=\Res_Q P_i=P_i (Q).$$ Since $S_Q$ takes projectives to
projectives, the resolution
$$ 0\rightarrow  P_n (Q) \rightarrow \cdots\rightarrow
P_0 (Q) \rightarrow  (\res^G_P M) (Q)\rightarrow  0$$ is a finite
projective resolution of $(\res^G_P M)(Q)=M(Q)$ as an
$R[N_P(Q)/Q]$-module. This gives that $M(Q)$ is projective as an
$R[N_P(Q)/Q]$-module.
\begin{lemma}
 For every $p$-group $Q$, there is a $p$-Sylow subgroup $P$
  of $G$ such that $N_P (Q)$ is a $p$-Sylow subgroup
of $N_G (Q)$.
\end{lemma}
\begin{proof}
Let $S$ be a $p$-Sylow subgroup of $N_G(Q)$, and pick a $p$-Sylow subgroup $P$ of $G$ containing $S$.  Since $N_P(Q) = N_G(Q) \cap P$ is a $p$-subgroup of $N_G(Q)$, we have $|N_P(Q)| \leq |S|$. But $S \leq P$ and $S \leq N_G(Q)$ implies $S \leq N_P(Q)$. Therefore $S=N_P(Q)$.
\end{proof}
 We can assume $P$ is a $p$-Sylow subgroup which
has this property. Then, by the $p$-local version of Rim's theorem,
we can conclude that
$M(Q)$ is projective as an $R[N_G(Q)/Q]$-module.
Now, consider the map
$$ \psi = (\psi_Q)\colon \bigoplus_{Q \in \Is(\G _G),\ l(Q)=s} E_Q\circ
\Res_Q M\rightarrow  M$$ where $\psi _Q\colon E_Q\circ
\Res_Q M\rightarrow  M$ is the map adjoint to the
identity map $\id\colon \Res_Q M\rightarrow \Res_Q M$. For every $K
\in \cF$ with $l(K)=s$, the induced map
$\psi (K)$
 is an isomorphism. This is because
$$(E_Q\circ \Res_Q M )(K)=\Res _K E_Q \Res _Q M= S_K E_Q \Res _Q
M\cong M(K)$$ if $K$ is conjugate to $Q$ and zero otherwise. So, we
have $l(\coker \psi )<s$. Therefore, there is a finitely generated
projective $\RG _G$-module $P$ with $l(P)<s$, and a map $\alpha
\colon P\rightarrow M$ such that $\psi \oplus\alpha$ is surjective.
If $K$ is the kernel of $\psi \oplus\alpha$, we get an exact
sequence of $\RG _G$-modules
$$ 0\rightarrow K\rightarrow P\oplus\bigoplus_{Q \in \Is(\G _G),\
l(Q)=s}E_Q\circ \Res_Q M\rightarrow M\rightarrow  0$$ where the
middle term is projective as an $\RG _G$-module, and $l(K)< s$. Note
that $\res^G_PK$ must have a finite projective resolution by
\cite[Lemma 11.6]{lueck3}. So, by induction, $K$ has a finite
projective resolution, and hence $M$ has a finite projective
resolution as well.
\end{proof}

\begin{remark} The inductive argument we use in the above proof
is similar to the argument used by L\" uck to prove
Proposition 17.31 in \cite{lueck3}.  By this result, any module $M$ over a finite EI-category $\G$ which has a finite projective resolution, admits a resolution of length $\leqslant l(M)$ provided that $M(x)$ is $R$-projective for all $x\in \Ob(\G)$. \qed
\end{remark}

It isn't clear to us how to generalize Theorem \ref{thm:rim_orbit} to integer   coefficients. For $R = \bZp$, the following example shows that the result does not hold for an arbitrary family $\cF$.

\begin{example}\label{example three-four}
 Let $G=S_5$ and $R=\bZ _{(2)}$
 , and  take $\cF$ as the
family of all $2$-subgroups and $3$-subgroups in $G$. Consider the
$\RG _G$-module $M=\uR{G/(C_2 \times C_3)}$ where $C_2$ and $C_3$
are as in Example \ref{exmp:restrictionformula}. It is clear that
the restriction of $M$ to a $2$-Sylow subgroup is projective (since its restriction to $H=S_4$ is already projective), but
$M$ does not have a finite projective resolution as an $\RG _G$-module. 

To see this, suppose that $M$ has a finite projective
resolution $ \PP \twoheadrightarrow M$. Then, $ \PP (C_3)$ will be a
finite projective resolution for $M(C_3)$ over $R[N_G (C_3)/C_3]$.
This is because $C_3=\la (123)\ra$ is a maximal subgroup in $\cF$.
This implies
$$M(C_3 )\cong R [ S_3\times C_2 / C_3\times C_2 ] \cong R[C_2]$$ is
projective as an $R[N_G(C_3)/C_3]$-module. But,
$$R[N_G (C_3)/C_3 ]=R[S_3\times C_2/C_3]\cong R[C_2\times C_2],$$
and it is clear that $R[C_2]$ is not projective as an $R[C_2\times
C_2]$-module. So, $M$ does not have a finite projective resolution.  \qed
\end{example}

On the other hand, the following holds for modules over orbit
categories:

\begin{proposition}\label{prop:almost Rim} Let $G$ be a finite group,
and $\cF $ be a family of subgroups of $G$. Then, a $\ZG
_G$-module $M$ has a finite projective resolution  if and only if
$\bZp \otimes _{\bZ} M$ has a finite projective resolution over
$\bZp \G _G$, for all primes $p$ dividing the order of $G$.
\end{proposition}

The proof of this statement follows from Propositions
\ref{prop:p-local decomposition} and \ref{prop:projectives} in the
next section.  We end this section with some corollaries of Theorem
\ref{thm:rim_orbit}.

\begin{corollary} Let $G$ be a finite group and $R=\bZp$. Suppose
that $\cF $ is a family of $p$-subgroups. Then, $\uR{G/H}$ has a
finite projective resolution over $\RG _G$ if a $p$-Sylow subgroup
of $H$ is included in $\cF$.
\end{corollary}

\begin{proof} If a $p$-Sylow subgroup of $H$ is in $\cF$, then
$\res ^G _P \uR{G/H}$ is a free $\RG _P$-module for any $P \in \Syl
_p (G)$. So, by Theorem \ref{thm:rim_orbit}, it has a finite
projective resolution.
\end{proof}

As a special case of this corollary, we obtain the following  known result 
 (see  \cite[6.8]{bouc8}, \cite[2.5 and p.~296]{symonds2}, \cite{jackowski-mcclure-oliver2}, \cite{grodal1}).

\begin{corollary}
Let $G$ be a finite group and $R=\bZp$. Then, $\un{R} $ has a finite
projective resolution over $\RG _G$ relative to the family of all
$p$-subgroups of $G$.
\end{corollary}

\begin{proof} This follows from $\un{R} = \uR{G/G}$.
\end{proof}


\section{Mackey structures on $\Ext^*_{\RG_G}(M,N)$}
\label{sect:Mackey}

The notation and results of the previous sections will now be used to establish some structural and computational facts about the $\Ext$-groups over the orbit category.  Our main sources are Cartan-Eilenberg \cite{cartan-eilenberg1} and  tom Dieck \cite[\S II.9]{tomDieck2}
(see also \cite{jackowski-mcclure-oliver2}, \cite{grodal1}).

We have seen that the category of right $\RG$-modules has enough
projectives to define the bifunctor
$$\Ext^*_{\RG}(M,N)=H^*(\Hom_{\RG}(\PP,N))$$
 via any projective
resolution $\PP\twoheadrightarrow M$  (see \cite[Chap.~III, \S
17]{lueck3}, \cite[Chap.~III.6]{maclane1}). The following property
is also useful (see  L\"uck \cite[17.21]{lueck3}).

\begin{lemma}\label{lem:E_x property}
If $\G$ is a 
free EI-category, then
$\Ext^*_{\RG}(E_xM,N)\cong\Ext^*_{R[x]}(M,\Res_xN)$.
\end{lemma}

\begin{proof}  
Take a projective resolution $\PP$ of $M$.  Since $\G$ is free, the extension functor $E_x$ is exact \cite[16.9]{lueck3}. In addition,  $E_x$ preserves projectives and is adjoint to the restriction functor $\Res_x$ by Proposition \ref{prop: induction-restriction}.
Therefore
$$\cdots\rightarrow E_x P_n \rightarrow
\cdots \rightarrow E_xP_1\rightarrow E_xP_0\rightarrow
E_xM\rightarrow 0$$ 
is a projective resolution of $E_xM$, and applying
$\Hom$ over the orbit category gives
\begin{eqnarray*}
\Ext^n_{\RG}(E_xM,N)&=& H^n(\Hom_{\RG}(E_x\PP, N))\\
&\cong& H^n(\Hom_{R[x]}(\PP, \Res _x N))=\Ext^n_{R[x]}(M,\Res_xN). \qedhere\end{eqnarray*}
\end{proof}

In the rest of  this section, we assume that $\G_G = \OrG$ for a
finite group $G$, where $\cF$ is a family of subgroups in $G$. 
Note that $\G_G$ is both finite and free as an EI-category.
 If there are
two groups $H \leq G$, we use the notations $\G _G = \OrG$  for the orbit category with respect to the family $\cF$,  and $\G
_H =\OrH$  for the orbit category with respect to the family  $\cF_H = \{ H \cap K \vv K \in \cF\}$. 

\begin{proposition}
\label{prop:Ext is finite} Let $M$ and $N$ be two $\ZG_G$-modules, where $M(H)$ is $\bZ$-torsion free for all $H \in \cF$.
Then for every $n>l(M)$, the groups $\Ext^n_{\ZG_G}(M, N)$ are
finite abelian, with exponent dividing some power of $|G|$.
\end{proposition}

\begin{proof} This follows from the Lemma \ref{lem:E_x property},
Proposition \ref{prop:resolution}, and the corresponding result for
modules over finite groups.
\end{proof}

Note that the $\Ext$-groups in lower dimensions are not finite in
general. But, it is still true  in all dimensions that the
$\Ext$-groups over $\ZG_G$ vanish if and only if they vanish over
$\bZp\G_G$, for all primes $p$. To see this, we note that tensoring over $\bZ$ with $\bZp$ preserves exactness, and hence
\eqncount
\begin{equation}\label{p-localization}
\Ext^n_{\ZG_G}(M, N) \otimes_\bZ \bZp \cong 
\Ext^n_{\bZp \G_G}(M\otimes_\bZ \bZp, N\otimes_\bZ \bZp).
\end{equation}
We also have the following:

\begin{proposition}\label{prop:p-local decomposition}
Let $M$ and $N$ be two $\ZG_G$-modules, where $M(H)$ is $\bZ$-torsion free for all $H \in \cF$. Then, for every $n> l(M)$,
there is an isomorphism $$ \Ext^n _{\ZG_G } (M, N ) \cong \bigoplus
_{p \mid |G|} \Ext^n _{\bZp \G_G } ( M_p , N_p)$$ where $M_p=\bZp
\otimes _{\bZ } M$ and $N_p=\bZp \otimes _{\bZ } N$.
\end{proposition}

\begin{proof} From Proposition \ref{prop:Ext is finite} we know that 
$ \Ext^n _{\ZG_G } (M, N )$ is a finite abelian group with exponent dividing some power of $|G|$, when $n> l(M)$. Now the 
 flatness of $\bZp$ over $\bZ$ implies as above that $ \Ext^n _{\ZG_G } (M, N )$ is the direct sum of its $p$-localizations, for all $p \mid |G|$. We then apply the isomorphism (\ref{p-localization}).
 \end{proof}

To complete the proof of Proposition \ref{prop:almost Rim}, we also
need the following  standard result in homological algebra
 (see \cite[Chap.~VI, 2.1]{cartan-eilenberg1} for the case of modules over rings):

\begin{proposition}\label{prop:projectives}
A right $\RG_G$-module $M$ admits a finite projective resolution if
and only if there exists an integer $\ell_0\geqslant 0$ such that
$\Ext^n_{\RG_G}(M,N)=0$, for all $n>\ell_0$ and all right
$\RG_G$-modules $N$.
\end{proposition}
\begin{proof}
If $M$ admits a finite projective resolution of length
$k$, then $\Ext^n_{\RG_G}(M,N)=0$ for $n > k$ and any $\RG_G$-module $N$. Conversely,
if $\Ext^n_{\RG_G}(M,N)=0$ for $n>\ell_0$ and any $N$, then consider
the kernel $Z_m$ of the boundary map $\bd _{m}: P_m \to P_{m-1}$ in
the projective resolution $\PP$ of $M$. It follows that $$\Ext^1
_{\RG_G}(Z_m ,N)\cong \Ext _{\RG_G} ^{m+2} (M, N )=0$$ for any $\RG_G$-module $N$,
provided $m+2>\ell_0$, and so $Z_m$ is projective if we take $m= \ell _0
-1$. This gives a finite projective resolution of length $\ell _0$
over $\RG_G$.
\end{proof}

We now recall the definition of a Mackey functor (following
Dress \cite{dress2}). Let $G$ be a finite group and $\cD(G)$ denote
the Dress category of finite $G$-sets and $G$-maps.  A bivariant functor
$$M=(M^*,M_*)\colon \cD(G) \rightarrow \RMod$$
consists of a contravariant functor
$$ M^*\colon\cD(G)\rightarrow \RMod$$ and a covariant
functor
$$M_*\colon \cD(G)\rightarrow \RMod.$$ The functors
are assumed to coincide on objects.  Therefore, we write
$M(S)=M_*(S)=M^*(S)$ for a finite $G$-set $S$.  If $f\colon
S\rightarrow  T$ is a morphism, we often use the notation
$f_*=M_*(f)$ and $f^*=M^*(f)$. If $S = G/H$ and $T=G/K$ with $H\leq
K $ and $f\colon G/H \to G/K$ is given by $f(eH)=eK$, then we use the
notation $f_* = \ind_H^K$ and $f^* = \res_H^K$.

\begin{definition}[{Dress \cite{dress2}}]
A bivariant functor is called a \emph{Mackey functor} if it has the following
properties:
\begin{enumerate}\addtolength{\itemsep}{0.2\baselineskip}
\item [(M1)]For each pullback diagram $$\xymatrix{X\ar[r]^h\ar[d]_g
&Y\ar[d]^k\\S\ar[r]_f&T }$$ of finite $G$-sets, we have $h_*\circ
g^*=k^*\circ f_*$.
\item [(M2)] The two embeddings $S\rightarrow  S\disjointunion
T\longleftarrow  T$ into the disjoint union define an isomorphism
$M^*(S\disjointunion T)\cong M^*(S) \oplus M^*(T)$.
\end{enumerate}
\end{definition}

\begin{remark} There is a functor $\cO(G) \to \cD(G)$ defined on
objects by $H \mapsto G/H$ for every subgroup $H \leq G$, and as the
identity on morphism sets. By composition, any contravariant functor $\cD(G) \to \RMod$ gives a
right $\RG_G$-module,  with respect to any given family of subgroups $\cF$ of $G$.

In the  statement of Theorem \ref{thm:Mackey structure}  we will use the examples
$\uR{S}\colon \cD(G) \to \RMod$, defined in (\ref{basic G-set modules}) for any finite $G$-set $S$.
\end{remark}

The following example and lemma will be used in the proof of Theorem \ref{thm:filtered mislin_orbit}.

\begin{example}\label{ex: coMackey}
Let  $Q \in \cF$ and let $V$ be a right $R[W_Q(Q)]$-module, where $W_G(Q) = N_G(Q)/Q$. Then we define a bivariant functor $D_Q(V)\colon \cD(G) \to \RMod$ on objects by setting
$$ D_Q(V)(S) = \Hom_{R[W_G(Q)]}(R[S^Q],V)$$
for any finite $G$-set $S$. For any $G$-map $f\colon S \to T$ we have a $W_G(Q)$-map $f^Q\colon S^Q \to T^Q$, which induces a homomorphism
$$f^*\colon  \Hom_{R[W_G(Q)]}(R[T^Q],V) \to  \Hom_{R[W_G(Q)]}(R[S^Q],V)$$
by composition. To define the covariant map $f_*$, let $\varphi_S\colon R[S^Q] \to V$ be an $R[W_G(Q)]$-homomorphism, and define $f_*(\varphi_S) = \varphi_T$ by
$$f_*(\varphi_S)(t) = \varphi_T (t)= \sum_{s \in S^Q, f(s) = t} \varphi_S(s)$$
It is not hard to verify that $D_Q(V)$ is actually a Mackey functor. The axiom (M1) follows because the $Q$-fixed sets in a pull-back diagram of $G$-sets give again a pull-back diagram. The axiom (M2) is immediate.
\end{example}
\begin{definition}\label{def: coresolution}
For any $\RG_G$-module $N$, we define 
$ DN=\sum _{Q\in \Is(\G_G)} D_Q (N(Q))$
and define $j\colon N \to DN$  as the direct sum of the adjoints of 
$\id \colon N(Q)\to N(Q)$, for each
$Q \in \Is(\G_G)$. Let $CN$ denote the cokernel of $j$.
For $k \geq 0$, define inductively $C^0N=N$ and $C^kN=C(C^{k-1} N)$, together with the induced maps $C^{k} \to DC^{k}$.
\end{definition}

Here is a dual construction to the E-resolution given in \cite[17.13]{lueck3}.
\begin{lemma}\label{lem: coresolution}
For any $\RG_G$-module $N$, the finite length sequence 
$$0 \to N \xrightarrow{j} DN \to DCN \to \dots \to DC^mN \to 0$$
is an exact coresolution of Mackey functors, for some $m\geq 0$. 
\end{lemma}
\begin{proof}
For any $\RG_G$-module $N$, the map $j\colon N \to DN$ defined above is injective, so we have a short exact sequence
$$0 \to N\xrightarrow{j} DN \to CN \to 0.$$
Iterating the above process, we obtain
$$0 \to CN \to DCN \to C^2N \to 0$$
and so on. By slicing, we get an exact sequence, or coresolution:
$$ 0 \to N \xrightarrow{j} DN \to DCN \to \cdots \to DC^{k-1} N \to DC^{k}N\to \cdots$$
 When $N$ is a $\RG_G$-module of a finite length, which is the
case in our situation, this coresolution has a finite length. 
To check this, we use the definition of $D_Q(V)$ in Example \ref{ex: coMackey} 
to get
$$D_Q (V) (K)=\Hom _{R[W_G(Q)]} ( R[(G/K)^Q], V)$$
for any $R[W_G(Q)]$-module $V$.
Therefore $D_Q (V) (K)$ is only nonzero for $(Q) \leq (K)$, and at $(Q)=(K)$ the $R[W_G(K)]$-module $D_Q (V )(K)$
is isomorphic to $V$, via the isomorphism $W_G(Q) \cong W_G(K)$ induced by conjugation.
This shows that the length of the module $C^k N$ is properly smaller than the length of $C^{k-1} N$ for all $k\geq 1$. 
\end{proof}

We will prove Theorem \ref{thm:mislin_orbit} by showing that $H\mapsto
\Ext^*_{\RG_H}(M,N)$ has a cohomological Mackey functor structure
which is conjugation invariant.  
First
we describe the Mackey functor structure on $\Hom_{\RG_{?}}(M,N)$.

\begin{theorem}
\label{thm:Mackey structure} For a right $\RG_G$-module $M$ and a
Mackey functor $N$, let $$\uHom(M,N)\colon \cD(G) \to \RMod$$ denote
the function defined by $S\mapsto \Hom_{\RG_G}(M\otimes_R\uR{S},N)$
for any finite $G$-set $S$. Then $\Hom_{\RG_?}(M,N)$ inherits a
Mackey functor structure.
\end{theorem}

\begin{proof}
We will first define the induction and restriction maps to see that
$\uHom(M,N)$ is a bifunctor. For $f\colon S\rightarrow  T$ a
$G$-map,  the \emph{restriction} map
$$f^*\colon  \Hom_{\RG_G}(M\otimes_R\uR{T},N)\rightarrow
\Hom_{\RG_G}(M\otimes_R\uR{S},N)$$ is the composition with
$M\otimes_R \uR{S}\xrightarrow{\id\otimes\tilde{f}}M\otimes_R
\uR{T}$ where $\tilde{f}$ denotes  is the linear extension of the
map induced by $f$. Since the functors $\uR{S}$ satisfy axiom (M2),
so does $\uHom(M,N)$.

For $f\colon S\rightarrow  T$ a $G$-map, we  define  the
\emph{induction} map
$$f_*\colon  \Hom_{\RG_G}(M\otimes_R\uR{S},N)\rightarrow
\Hom_{\RG_G}(M\otimes_R\uR{T},N)$$ in the following way: let
$\varphi _S\colon M\otimes_R \uR{S}\rightarrow  N$ be given. We
will describe the homomorphism $\varphi _T=f_*(\varphi _S)$.
$$\varphi _T(V)(x\otimes\alpha)=F_*\Bigl (\varphi _S(U)(F^*(x)\otimes
\beta ) \Bigr )$$ for $x\in M(V)$ and $\alpha\colon V\rightarrow  T$,
where $U$, $\beta $ and $F$ are given by the pull-back
$$\xymatrix{U\ar[r]^\beta \ar[d]_F&S\ar[d]^f\\V\ar[r]_\alpha&T}$$
It is easy to check that this formula for $\varphi _T$ gives an
$\RG_G$-homomorphism, using the assumption that $N$ is a Mackey functor.

We need to check axiom (M1) for $\uHom(M,N)$. For a given  pull-back
square
$$\xymatrix{X\ar[r]^h\ar[d]_g&Y\ar[d]^k\\S\ar[r]_f&T}$$
we need to show that $h_*\circ g^*=k^*\circ f_*$. Let $\gamma\colon
V \to Y$ be any $G$-map, and consider the extended pull-back diagram
$$\xymatrix{U\ar[r]^\delta \ar[d]_F&X\ar[r]^g\ar[d]_h&S\ar[d]^f\\
V\ar[r]_\gamma&Y\ar[r]_k&T}$$ The maps $\alpha = k \circ \gamma$ and
$\beta = g\circ \delta$ may be used to compute $f_*(\varphi_S)$ as
above, and the left-hand square may be used to compute $h_*$.

For any element $\varphi_S\colon M\otimes_R \uR{S}\to  N$, we have
\begin{eqnarray*}
(k^*\circ f_*(\varphi _S))(V)(x\otimes\gamma)&=&(f_*(\varphi _S)
\circ (\id\otimes k))(V)(x\otimes\gamma)\\
&=&f_*(\varphi _S)(V) (x\otimes (k\circ\gamma))\\
&=&F_*(\varphi _S(U) (F^*(x)\otimes (g\circ \delta ))
\end{eqnarray*}
for any $x\in M(V)$ and $\gamma\colon V\to  Y$.
On the other hand,
\begin{eqnarray*}
(h_*\circ g^*(\varphi _S))(V)(x\otimes\gamma)&=&F_*((g^*\varphi _S)(U)
(F^*(x)\otimes\delta ))\\
&=&F_*(\varphi _S(U) (F^*(x)\otimes (g\circ\delta ))
\end{eqnarray*}
for any $x\in M(V)$ and $\gamma\colon V\to  Y$, so the formula (M1)
is verified.
\end{proof}

As an immediate consequence, for any subgroup $H\leq K$ the $G$-map
$f\colon G/H \to G/K$ induces  a restriction map $$\res^K_H\colon
\Hom_{\RG_K}(M,N) \to \Hom_{\RG_H}(M,N)$$ defined as the composition
of the  map
$$f^*\colon \Hom_{\RG_G}(M\otimes_R\uR{G/K},N)\to \Hom_{\RG_G}
(M\otimes_R\uR{G/H},N)$$ with the `Shapiro' isomorphisms:
$$\Hom_{\RG_G}(M\otimes_R\uR{G/H},N) \cong \Hom_{\RG_H}(M,N)$$
and
$$\Hom_{\RG_G}(M\otimes_R\uR{G/K},N) \cong \Hom_{\RG_K}(M,N)$$
given
by \cite[Cor.~2.12]{symonds2} and the adjointness  property (compare \cite[Lemma 2.8.4]{benson1}). Similarly,
we have the induction map
$$\ind_H^K\colon  \Hom_{\RG_H}(M,N) \to  \Hom_{\RG_K}(M,N)$$
defined by composing the Shapiro isomorphisms with $f_*$.

\begin{remark}
Since $\Res_H^G$ preserves projectives, we see that
$P \otimes_R \uR{G/H}$ is projective over $\RG_G$ whenever $P$
is projective over $\RG_G$ (check the categorical lifting
property directly or apply Lemma \ref{lem:ind-res formulas1}).
\end{remark}

\begin{proposition}
Let $\CC$ be a chain complex of right $\RG_G$-modules and $N$ be a
Mackey functor. Then, the cochain complex
$$\bC^*=\uHom(\CC,N)$$ with the differential $\delta\colon
\uHom(C_i,N)\rightarrow \uHom(C_{i+1},N)$ given by
$\delta(\varphi )=\varphi \circ\bd$ is a cochain complex of Mackey
functors.
\end{proposition}

\begin{proof}
We have seen that each $C^i=\uHom(C_i,N)$ is a Mackey functor
by Theorem \ref{thm:Mackey structure}.  We just need to show that
the coboundary maps are Mackey functor maps.  Given $f\colon
S\rightarrow  T$ we need to show the following diagram commutes:
$$\xymatrix{ \Hom_{\RG_G}(C_i\otimes
\uR{S},N)\ar[r]^{\delta_S}\ar@<2pt>[d]^{f_*}
&\Hom_{\RG_G}(C_{i+1}\otimes \uR{S},N)\ar@<2pt>[d]^{f_*}\\
\Hom_{\RG_G}(C_i\otimes \uR{T},N)\ar@<2pt>[u]^{f^*}\ar[r]^{\delta_T}
&\Hom_{\RG_G}(C_{i+1}\otimes \uR{T},N)\ar@<2pt>[u]^{f^*}}$$ The
proof of commutativity for $f^*$ is easy.  In this case, it follows
from the commutativity of the following diagram:
$$\xymatrix{ C_i\otimes \uR{S}\ar[d]^{\id\otimes f}&C_{i+1}\otimes
\uR{S}
\ar[l]_{\bd\otimes \id}\ar[d]^{\id\otimes f}\\
C_i\otimes \uR{T}&C_{i+1}\otimes \uR{T}\ar[l]_{\bd\otimes \id}}$$
For $f_*$ we check the commutativity directly: let $\varphi _S\colon
C_i\otimes \uR{S} \rightarrow N$ be an $\RG_G$-map. For $x\in
C_{i+1}(V)$ and $\alpha\colon V\rightarrow  T$, we have
\begin{eqnarray*}
[(\delta_T\circ f_*)\varphi _S](x\otimes\alpha)&=&(f_*\varphi _S)(\bd
x\otimes\alpha)\\
&=&F_*[\varphi _S(F^*(\bd x)\otimes\beta )]
\end{eqnarray*}
where
$$\xymatrix{U\ar[r]^\beta \ar[d]_F&S\ar[d]^f\\
V\ar[r]_\alpha&T}$$ on the other hand,
\begin{eqnarray*}
[(f_*\circ\delta_S )\varphi _S](x\otimes\alpha)&=&F_*[(\delta_S
\varphi _S)(F^*(x)\otimes\beta )]\\
&=&F_*[\varphi _S\circ(\bd\otimes \id)(F^*(x)\otimes\beta )]\\
&=&F_*[\varphi _S(\bd F^*(x)\otimes\beta )]
\end{eqnarray*}
since $\bd F^*=F^*\bd$, we are done.
\end{proof}

\begin{corollary}
Let $M$ be an $\RG_G$-module and $N$ be a Mackey functor.  Then,
$$\uExt^*(M,N)$$ has a Mackey functor structure.  As a Mackey
functor $\uExt^*(M,N)$ is equal to the homology of the
cochain complex of Mackey functors $\uHom(\PP,N)$ where $\PP$
is a projective resolution of $M$ as an $\RG_G$-module.
\end{corollary}
\begin{proof} To compute the $\Ext$-groups, 
note that $S\mapsto \PP\otimes_R \uR{S}$ is a projective resolution
of the module $S\mapsto M \otimes_R \uR{S}$, for every finite
$G$-set $S$. 
\end{proof}
\begin{remark}
It follows that a version of the Eckmann-Shapiro isomorphism
$$\Ext^*_{\RG_G}(M\otimes \uR{G/H},N)\cong \Ext^*_{\RG_H}(\Res_H^G
M,\Res_H^G N)$$ holds for the $\Ext$-groups over the orbit category (compare  \cite[2.8.4]{benson1}).
\end{remark}

\begin{remark} If $N$ is a Green module over a Green ring $\cG$, then
the Mackey functor $\uExt^*(M,N)$ also inherits a Green module
structure over $\cG$. The basic formula is a pairing $$\cG(S) \times
\Hom_{\RG_?}(M\otimes_R\uR{S},N) \to
\Hom_{\RG_?}(M\otimes_R\uR{S},N)$$ induced by the Green module
pairing $\cG \times N \to N$. For any $z\in \cG(S)$, $x\in M(V)$,
and $\alpha \colon V \to S$, we define
$$(z\cdot \varphi_S)(V)(x\otimes \alpha) = \alpha^*(z)\cdot
\varphi_S(V)(x\otimes \alpha)$$
for any $\varphi_S(V)\colon M(V) \otimes_R R\Mor(S, V) \to N(V)$.
The check that this pairing gives a Green module structure is left
to the reader. \qed
\end{remark}

\section{The proof of Theorem C}\label{sec: thmC}
The main purpose of this section to prove the following theorem.
\begin{theorem}
\label{thm:mislin_orbit} Let $G$ be a finite group, $R=\bZp$, and
$\cF$ be a family of subgroups in $G$. Suppose $H\leq G$ controls
$p$-fusion in $G$. Then,
$$ \res_H^G\colon \Ext^n_{\RG_G}(M,N)\rightarrow
\Ext^n_{\RG_H}(\res_H^GM, \res_H^GN)$$ is an isomorphism for $n\geqslant
0 $, provided that $M$ is an $\RG_G$-module and $N$ is a
cohomological Mackey functor satisfying the condition that $C_G(Q)$
acts trivially on $N(Q)$ and $M(Q)$ for all $p$-subgroups $Q\leq H$, with  $Q \in\cF $.
\end{theorem}

Certain Mackey functors (called \emph{cohomological}) are computable
by restriction to the $p$-Sylow subgroups and the conjugation action
of $G$ (see \cite[Chap.~XII, \S 10]{cartan-eilenberg1}, \cite{lam1}).

If $H\leq G$ is a subgroup, and $n\in N_G(H)$ then the $G$-map
$f\colon G/H \to G/H$ defined by $f(eH) = nH$ has an associated
conjugation homomorphism $c_n(h) = n^{-1}hn \in H$, for all $h\in
H$. For an arbitrary $\RG_G$-module $M$, the induced maps $f^*$ need
not be the identity on $M(G/H)$ even if $c_n = \id$ (e.g.~if $n \in
C_G(H)$).

\begin{definition}\label{def: conjugation-invariant}
We say a Mackey functor is \emph{cohomological} (over $\cF$) if $$\ind _H^K\res^K_H
(u)=|K\colon H|\cdot u$$ for all $u\in M(K)$, and all $H \leq K$ (for all $K \in \cF$). An
$\RG_G$-module $M$ with respect to a family $\cF$ is called
\emph{conjugation invariant} if $C_G(Q)$ acts trivially on $M(Q)$
for all $Q \in \cF$. A Mackey functor is called \emph{conjugation
invariant} if it is conjugation invariant as a functor over the
corresponding orbit category.
\end{definition}
The following lemma will be used in the proof of Theorem \ref{thm:filtered mislin_orbit}.
\begin{lemma}\label{lem: coMackey}
Let  $Q \in \cF$ and let $V$ be a right $R[W_Q(Q)]$-module. If $\cF$ is a family of $p$-subgroups, and $R=\Fp$,  then  $D_Q(V)\colon \cD(G) \to \RMod$ is a cohomological Mackey functor over $\cF$. If $C_G(Q)$ acts trivially on $V$, then $D_Q(V)$ is conjugation invariant.
\end{lemma}
\begin{proof}
Since all subgroups in $\cF$ are $p$-groups, for the first part we only need to
show that the composite $\ind _H^K\res^K_H
(u)=p\cdot u$, for $K \in \cF$ and $H \leq K$ a normal  of index $p$.

Let  $f\colon G/H \to G/K$ be
the $G$-map given by $gH\mapsto gK$. Consider the induced map
$f^Q \colon (G/H)^Q \to (G/K)^Q$. Take $ t \in (G/K)^Q$. If there
is no $s\in (G/H)^Q$ such that $f(s)=t$, then the transfer is
trivially zero. Suppose that there is at least one element
$s=gH$ which is fixed by $Q$ and maps to $ t=gK$. Let
$k_1,\dots ,k_p$ be coset representatives of $H$ in $K$.
Since $k_i$ normalizes $H$,  the element $gk_i H\in (G/H)^Q$ for each $i$. Therefore, there are exactly $ p$ different $s\in (G/H)^Q$
that map to $t$. It follows that $f_*\circ f^*$ is multiplication by $p$, as required. Since we are working here over the finite field $\Fp$, all the composites
$f_*\circ f^*=0$.

We now show that $D_Q(V)$ is conjugation invariant if $C_G(Q)$ acts trivially on $V$. In other words, we claim that for all
$K \in \cF$, the centralizer $C_G(K)$ acts trivially on $\Hom  _{R]W_G(Q)]} (R[G/K]^Q, V)$.
Consider the way the action is defined: let
$c \in C_G(K)$ and $\varphi\colon  R[G/K]^Q \to V$ be an $R[W_G(Q)]$-map.
Then $(c\varphi) (gK)=\varphi(gcK )$. On the other hand since $ gK\in (G/K)^Q$,
we have $Q^g \leq K$. So, $c$ centralizes $Q^g$. This means
$gcg^{-1}$ centralizes $Q$ and hence acts trivially on $ V$.
This gives
$$\varphi(gcK)=\varphi( gcg^{-1} gK)=gcg^{-1} \varphi(gK)=\varphi(gK)$$
Therefore $(c\varphi)(gK)=\varphi(gK)$ for all $ gK$. This shows that
$c \in C_G(K)$ acts as the identity on $\Hom  _{R[W_G(Q)]} (R[G/K]^Q, V)$.
\end{proof}

The cohomological and conjugation properties are inherited by the $\Ext$-functors.
\begin{proposition}\label{prop:ext-properties}
Let $M$ and $N$ be $\RG_G$-modules relative to some family $\cF$.
\begin{enumerate}
\item If $N$ is a cohomological Mackey functor over $\cF$, then $\uExt^*(M,N)$ is
a cohomological Mackey functor over all subgroups $H \leq G$.
\item If both $M$ and $N$ are conjugation
invariant with respect to $\cF$, then $\uExt^*(M,N)$ is conjugation
invariant with respect to all subgroups $H\leq G$.
\end{enumerate}
\end{proposition}

\begin{proof}
We have seen that for $f\colon S\rightarrow  T$, the induced maps
$$\xymatrix{\Hom_{\RG_G}(M\otimes \uR{S},N)\ar@<2pt>[r]^{f_*}
&\Hom_{\RG_G}(M\otimes \uR{T},N)\ar@<2pt>[l]^{f^*}}$$ satisfy the
property that
\begin{eqnarray*}
[(f_*\circ f^* )\varphi _T](V)(x\otimes\alpha)&=&F_*[f^*(
\varphi _T)(U)(F^*(x)\otimes\beta )]\\
&=&F_*[\varphi _T(U)(F^*(x)\otimes(f\circ\beta ))]\\
&=&F_*[\varphi _T(U)( F^*(x)\otimes(\alpha\circ F))]\\
&=&(F_* \circ F^*) [\varphi _T (V)(x\otimes\alpha)]
\end{eqnarray*}
for all $x \in M(V)$ and $\alpha\colon V \to T$. In the last
equality we used the invariance of $\varphi_T $ with respect to the
$G$-map $F\colon U \to V$ (our notation comes from the definition of
$f_*$ above). Hence, if $f\colon G/H\rightarrow G/K$ and $F_*\circ
F^*$ is multiplication by $|K\colon H|$ (this follows from a count
of double cosets), then $f_*\circ f^*$ is also multiplication by
$|K\colon H|$.
\smallskip

Let $M$ and $N$ be conjugation invariant right $\RG_G$-modules, and
let $\PP$ be a projective resolution of $M$ over $\RG _G$. To show
that $\uExt^*(M,N)$ is conjugation invariant, it is enough to show
that the chain map induced by the conjugation action on 
$\uHom (\PP, N)$ is homotopy equivalent to the identity. 
We remark that the action of an element $c \in C_G(H)$ gives an automorphism  $J_c\colon \OrH \to \OrH$, and induces an $\RG_H$-module chain map $\PP(J_c)\colon \Res_H^G(\PP) \to \Res_H^G(\PP)$.

 If $f\colon
G/H\rightarrow G/H$ is given by $eH\mapsto cH$ where $c\in C_G(H)$,
then for each degree $i$,
$$f^*_i \colon \Hom_{\RG_G}(P_i \otimes \uR{G/H} ,N) \rightarrow
\Hom_{\RG_G}(P_i \otimes \uR{G/H} ,N)$$ is given by
$$f_i^*(\varphi _S)(U)(x\otimes\alpha)= \varphi _S (U) ( x\otimes f\circ \alpha )$$
where $S= G/H$, $x\in P_i(U)$, and $\alpha\colon U \to G/H$ is a $G$-map.
In other words, $f^*_i = \Hom_{\RG_G}(\lambda_i, \id)$, where
$\lambda_i(x \otimes \alpha) = x\otimes f\circ \alpha$ defines a
chain map $$\lambda\colon \PP \otimes \uR{G/H} \to \PP \otimes \uR{G/H}.$$

We may assume that $U=G/K$ with $K  \in \cF$. Let $\alpha(eK) = gH$.
The conjugation action of $c\in C_G(H)$ on $M(U)$ or $N(U)$ is given
by the $G$-map $F\colon G/K \to G/K$, where $F(eK) = gcg^{-1}K$ and
$f\circ \alpha = \alpha\circ F$. We remark that $z:=gcg^{-1} \in
C_G(K)$, since $K \subseteq gHg^{-1}$, and that $\PP^*(F) = \PP(J_z)(K)$. Notice that
$$f_i^*(\varphi _S)(U)(x\otimes\alpha)=(\varphi _S (U) (x\cdot P_i^*(F)^{-1}\otimes \alpha))\cdot N^*(F),$$
showing that the maps $f^*_i$ are just given by the natural action  
maps of $c$ on the domain and range of the $\Hom$.
 Now observe that  $$\PP(J_z)\colon\Res_K^G(\PP) \to \Res_K^G(\PP)$$ is a chain map lifting $M(J_z)\colon \Res_K^G(M) \to \Res_K^G(M)$. Since
 $M$ is conjugation invariant, it follows that $\PP(J_z)\simeq \id$ by
 uniqueness (up to chain homotopy) of lifting in projective resolutions.
Therefore $\lambda_1 := \lambda\circ (\PP^*(F) \otimes \id)\simeq
\lambda$, and $f^* \simeq \Hom(\lambda_1, \id)$. But for all $x \in
P_i(U)$,  we have $$\Hom(\lambda_1, \id)(\varphi _S)(U)(x\otimes
\alpha)  = \varphi _S (U) ( x\cdot P_i^*(F)\otimes f\circ \alpha) =
(\varphi _S (U) ( x\otimes \alpha))\cdot N^*(F),$$ and hence
$f^*(\varphi _S) \simeq \varphi_S$, by the conjugation invariance of
$N$.
\end{proof}

\begin{definition} For any subgroup $H \leq G$, and any $\RG_G$-modules
$M$ and $N$, an element $\alpha \in \Ext^n_{\RG_H}(M, N)$ is called
\emph{stable} with respect to $G$ provided that
$$\res^H_{H \cap \leftexp{g}H}(\alpha) = \res^{\leftexp{g}H}_{H \cap
\leftexp{g}H} c_H ^g ( \alpha)$$ for any $g\in G$. The map $c_H ^g $
is the induced map $f_*$ where $f\colon G/H \to G/\leftexp{g}H$ is the
$G$-map given by $xH \to x g^{-1} (gH g^{-1})$.
\end{definition}

\begin{theorem}
\label{thm:stableisom} Let $R = \bZp$ and $G$ be a finite group. For
a right $\RG_G$-module $M$ and a cohomological Mackey functor
$N\colon  \cD(G) \to \RMod$, the restriction map $$\res^G_P\colon
\Ext^n_{\RG_G}(M,N) \to \Ext^n_{\RG_P}(M,N)$$ is an isomorphism for
$n\geqslant 0$ onto the stable elements, for any $p$-Sylow subgroup $P\leq
G$.
\end{theorem}

\begin{proof} By Proposition \ref{prop:ext-properties}(i), $\Ext^n_{\RG_?}(M,N)$ is a cohomological Mackey functor. Now the result follows
 (as in \cite[2.2]{symonds1}) from the stable element method of
Cartan and Eilenberg \cite[Chap.~XII, 10.1]{cartan-eilenberg1}.
\end{proof}

\begin{remark} Since $\uExt^*(M,N)$ is a cohomological Mackey functor,
it is a Green module over the trivial module $\un{R}$, considered as
a Green ring by defining $\ind_H^K\colon \un{R}(G/H) \to
\un{R}(G/K)$ to be multiplication by $|K: H|$ (see
\cite[Ex.~2.9]{lam1}). It follows that $\uExt^*(M,N)$ is computable
in the sense of Dress in terms of the $p$-Sylow subgroups (see
\cite[Ex.~5.9]{htw2007}).
\end{remark}

\begin{proof}[The proof of Theorem \ref{thm:mislin_orbit}]
Let $R = \bZp$ and $G$ be a finite group.  Let $H\leq G$ be a
subgroup which controls $p$-fusion in $G$.  For any  cohomological Mackey functor $F$, the restriction map $\Res_P^G$ maps surjectively to the stable elements in $F(P)$, for any $p$-Sylow subgroup $P \leq G$. If $H$ controls $p$-fusion in $G$, and $F$ is conjugation invariant, then all elements in $F(H)$ are stable and
$$\res^G_H\colon  F(G)\xrightarrow{\approx} F(H)$$ is an
isomorphism. 
 This follows by
a standard argument used to prove one direction of Mislin's theorem
in group cohomology (see, for example,   Symonds \cite[Theorem 3.5]{symonds1} or  Benson
\cite[Proposition 3.8.4]{benson1}). We apply Proposition \ref{prop:ext-properties} and this remark to the cohomological Mackey
functor $F=\uExt^n(M,N)$, and the proof is complete.
\end{proof}

In the next section we will need a variation of this result.

\begin{definition}
We say the $N$ is an \emph{atomic} right $\RG_G$-module of type $Q \in
\cF$, if  $N=I_Q (N(Q))
$ where $I_Q$ is the inclusion functor
introduced in Section \ref{sect:definitions}.
\end{definition}

\begin{theorem}\label{thm:filtered mislin_orbit}
 Let $G$ be a finite group, $R=\bZp$, and let
$\cF$ be a family of $p$-subgroups in $G$. Suppose $H\leq G$
controls $p$-fusion in $G$. Then, for $\RG_G$-modules $M$ and $N$,
$$ \res_H^G\colon \Ext^n_{\RG_G}(M,N)\rightarrow
\Ext^n_{\RG_H}(\res_H^GM, \res_H^GN)$$ is an isomorphism for
$n\geqslant 0 $, provided that $C_G(Q)$ acts trivially on $M(Q)$ and
$N(Q)$ for all $Q \in \cF$.
\end{theorem}

\begin{proof}
By the 5-lemma (using the filtration in \cite[16.8]{lueck3})  it is enough to prove
the statement for $N$ an atomic $\RG_G$-module of type $Q$.
Without loss of generality we can also assume that
$N(Q)$ is $R$-torsion free. 
To see this, observe that as an
$N_G(Q)/ Q C_G(Q)$-module, $N(Q)$ fits into a short exact sequence
$ 0 \to L \to F \to N(Q)\to 0$, where $F$ is a free $N_G(Q)/Q C_G(Q)$-module.
By taking inflations of these modules, we can consider the sequence
as a sequence of $N_G(Q)/Q$ -modules and apply the functor $I_Q$.
This shows that $N$ fits into a sequence $0 \to N'' \to N' \to N\to 0$,
where both $N'$ and $N''$ are conjugation invariant and atomic, with an $R$-torsion free module
at $Q$.

Since $\cF$ is family of $p$-groups, Lemma \ref{lem: coMackey} shows that the Mackey functor $D_Q(V)\otimes \Fp$ defined in Example \ref{ex: coMackey} is cohomological over $\cF$ and conjugation invariant. Let $N_p = N \otimes \Fp = N/pN$.
By Lemma \ref{lem: coresolution} we have
a finite length coresolution of Mackey functors
\eqncount
\begin{equation}\label{coresolution}
0 \to N_p \to DN_p \to DCN_p \to \dots \to DC^mN_p \to 0
\end{equation}
for some $m\geq 0$. 

Now we can apply the functors $\Ext^*_{\RG_?}(M, N_p)$ to the coresolution (\ref{coresolution}). By Proposition \ref{prop:ext-properties} the functors 
$\Ext^*_{\RG_?}(M, N_p)$ are cohomological Mackey functors which are conjugation invariant. Therefore
$$\Res_H^G \colon \Ext^*_{\RG_G}(M, N_p) \to \Ext^*_{\RG_H}(M, N_p)$$
is an isomorphism by Theorem \ref{thm:mislin_orbit} and the 5-lemma using the coresolution. Furthermore,  we have a short exact sequence
$$ 0 \to N/p^{k-1} \to N/p^k \to N/p \to 0,$$
for every $k \geq 1$, and hence by ``d\'evissage"  we conclude that 
\eqncount
\begin{equation}\label{devissage}
\Res_H^G \colon \Ext^*_{\RG_G}(M, N/p^k) \xrightarrow{\approx} \Ext^*_{\RG_H}(M, N/p^k)
\end{equation}
is an isomorphism, for every $k \geq 1$.
 To finish the proof it is enough to show that
$$\Res_H^G\colon \Ext^*_{\RG_G}(M, N)\otimes \Zphat \to \Ext^*_{\RG_H}(M, N)\otimes \Zphat$$
is an isomorphism. However, the complex
$$\Hom_{\RG_G}(\PP, N/p^k) = \Hom_{\RG_G}(\PP, N)\otimes \bZ/p^k$$ 
is a cochain complex of finitely-generated $R$-modules. By the universal coefficient theorem in cohomology \cite[p.~246]{spanier}, we have an exact sequence
$$0 \to \Ext^n_{\RG_G}(M, N)\otimes \bZ/{p^k} \to  \Ext^n_{\RG_G}(M, N/p^k)
\to \text{Tor}_1^R( \Ext^{n+1}_{\RG_G}(M, N), \bZ/{p^k}) \to 0.$$
Since $\Zphat = \varprojlim \bZ/{p^k}$ and the inverse limit functor is left exact, we obtain an exact sequence
$$0 \to  \Ext^n_{\RG_G}(M, N)\otimes \Zphat \to  \varprojlim \Ext^n_{\RG_G}(M, N/p^k)
\to \varprojlim \text{Tor}_1^R( \Ext^{n+1}_{\RG_G}(M, N), \bZ/{p^k}).$$
Now we compare this sequence via $\Res_H^G$ to the corresponding sequence for the subgroup $H$, and use the d\'evissage isomorphisms (\ref{devissage}) on the middle term. This shows immediately that
$\Res_H^G$ is injective on the first term, for all $n \geq 0$. 
Since the functor $\text{Tor}^R_1$ is left exact, 
we get $\Res_H^G$ injective on the third term as well. But now a diagram chase shows that $\Res_H^G$ is surjective on the first term. 
 \end{proof}

\section{Chain complexes over orbit categories}
\label{sect:chain complexes}

In this section, we prove some theorems about chain complexes over
orbit categories. In particular, Proposition \ref{prop:adding-subtracting}, Proposition \ref{prop:mislin_chain}, and
Theorem \ref{thm:glueing}  will be used in the proof of Theorem A (see Section \ref{sect:construction-new}).  Most of
the results follow from Dold's theory of algebraic Postnikov
systems \cite{dold1}. 

As before, $G$ denote a finite group and $\cF
$ denote a family of subgroups of $G$. Throughout this section
$\G_G=\cO _\cF G$ and $R$ is a commutative ring. For chain complexes
$\CC$ and $\DD$, the notation $\CC \simeq \DD$ always means $\CC$ is
chain homotopy equivalent to $\DD$. For chain isomorphism the
standard notation is $\CC \cong \DD$. When we say $\CC$ is a
projective chain complex, we mean it is a chain complex of
projective modules (which also means that it is projective in the
category of chain complexes). A chain complex $\CC$ is  \emph{positive} if $C_i = 0$ for $i <0$.

We say that a chain complex $\CC$ over $\RG_G$ has \emph{finite homological dimension} (or $\hdim \CC$ is finite) if $\CC$ is positive, and there exists an integer $n$ such that $H_i(\CC) = 0$ for $i > n$. A chain complex $\CC$ is \emph{finite} if $\CC$ is positive, and there exists an integer $n$ such that $C_i = 0$ for $i >n$. We start with a well known observation about chain complexes.

\begin{lemma}
Let $\CC$ be a projective chain complex of $\RG_G$-modules which has
finite homological dimension.  Then, $\CC$ is homotopy equivalent to
a finite projective chain complex if and only if there is an integer $n$
such that $$\Ext^i_{\RG_G}(\CC,M)=0\,\, \textrm {for}\,\, i>n,$$ for
all $\RG_G$-modules $M$.
\end{lemma}
\begin{proof} See Cartan-Eilenberg \cite[Chap.~XVII, 1.4]{cartan-eilenberg1} for chain complexes over rings. A similar argument as in Proposition \ref{prop:projectives} gives the result over the orbit category.
\end{proof}

\begin{proposition}\label{lem:finiteness}
Let $\CC$ be a projective chain complex of $\ZG_G$-modules which has
a finite homological dimension. Suppose that $\bZp \otimes _{\bZ}
\CC$ is chain homotopy equivalent to a finite projective chain
complex for all $p \mid |G|$. Then, $\CC$ is chain homotopy
equivalent to a finite projective complex.
\end{proposition}

\begin{proof} Let $M$ be an $\RG_G$-module. Consider the hypercohomology
spectral sequence (see \cite[3.4.3]{benson2}):
$$E_2 ^{s,t} = \Ext ^s _{\ZG_G} (H_t (\CC ) ,
M)$$ which converges to $\Ext^* _{\ZG_G}(\CC,M)$. Since $\CC$ has
finite homological dimension, for all $i > \bigl (l(\G_G)+ \hdim \CC
\bigr )$, the group $\Ext ^i _{\bZ} (\CC, M)$ is a finite abelian
group with exponent dividing a power of $|G|$. Here $l(\G_G)$ is the length of the orbit category, as defined in \S \ref{subsection:resolutions}, and $\hdim\CC$ denotes the largest integer $n$ such that $H_n(\CC) \neq 0$.

In particular, there
is an integer $k$, independent from $M$, such that
$$\Ext^i _{\ZG_G}(\CC,M)\cong \bigoplus _{p \mid |G|}
\Ext^i _{\bZp\G_G}(\bZp \otimes _{\bZ}\CC ,M _p)$$ for all $i >k$.
Here $M_p=\bZp \otimes _{\bZ} M$. Now, since $\bZp \otimes _{\bZ}
\CC$ is homotopy equivalent to a finite projective complex for all
$p \mid |G|$, there is an $n$ such that $$ \Ext ^i _{\ZG_G} (\CC ,
M)=0$$ for all $i>n$ and for all $M$. The result follows from the
previous lemma.
\end{proof}

A chain complex version of Rim's theorem also holds.

\begin{proposition}\label{prop: finite projective dimension} Let $R=\bZp$ and $\CC$ be a  projective chain
complex over $\RG_G$ with finite homological dimension. Assume that $\cF$ is a family of $p$-subgroups.
Then, $\CC$ is homotopy equivalent to a finite projective complex if
and only if $\res ^G _P \CC$ is homotopy equivalent to a finite
projective complex for any $p$-Sylow subgroup $P$ of $G$.
\end{proposition}

\begin{proof} One direction is clear (and holds without assumption on the family $\cF$). Conversely, suppose that $\res ^G _P \CC$ is
homotopy equivalent to a projective complex with $\hdim =l$. Let $n$
be an integer bigger than both $l$ and $\hdim \CC $. Consider
$$\xymatrix{
\cdots\ar[r]&\res_P^GC_{n+1}\ar[r]&\res_P^GC_{n}\ar[r]^{\bd _{n}}
&\res_P^GC_{n-1}\ar[r]&\cdots\ar[r]&\res_P^GC_{0}\ar[r]&0}.
$$
We have
$$ \Ext ^1 _{\RG _P} ( \res ^G _P \Image (\bd _{n}) , M )
\cong \Ext_{\RG _P} ^{n+1} ( \res ^G _P \CC, M)=0,$$ for every $\RG
_P$-module $M$. This gives that $\res ^G _P \Image(\bd _{n})$ is
projective. By Theorem \ref{thm:rim_orbit}, we obtain that
$\Image (\bd _{n})$ has finite projective resolution. Thus, $\CC$ is
chain homotopy equivalent to a finite projective complex.
\end{proof}

We also prove a chain complex version of Theorem \ref{thm:filtered
mislin_orbit}. Recall the definition of \emph{conjugation invariant} $\RG_G$-modules given in (\ref{def: conjugation-invariant}).

\begin{proposition}\label{prop:mislin_chain}
Let $G$ be a finite group, and $\cF$ be a family of $p$-subgroups in
$G$. Suppose $H\leq G$ controls $p$-fusion in $G$ and  $R=\bZp$. Let
$\CC^H$ be a positive projective chain complex of $\RG_H$-modules such that  the homology groups $H_i(\CC^H)$ are conjugation invariant right $\RG_H$-modules, for every $i \geqslant 0$. 
Then, the
following holds:
\begin{enumerate}
\item There exists a positive projective chain complex $\CC ^G $ of $\RG
_G$-modules such that $\res^G _H \CC ^G $ is homotopy equivalent to
$\CC ^H$.
\item If $\CC ^H$ is homotopy equivalent to a finite projective complex,
then $\CC ^G$ is also homotopy equivalent to a finite complex.
\end{enumerate}
\end{proposition}

For the proof we will need the theory of algebraic Postnikov systems
due to Dold \cite[\S 7]{dold1}. According to this theory, given a
positive projective chain complex $\CC$, there is a sequence of positive projective
chain complexes $\CC(i)$ indexed by positive integers such that
$f\colon \CC\rightarrow \CC(i)$ induces a homology isomorphism for
dimensions  $\leq i$. Moreover, there is a tower of
maps $$\xymatrix{&\CC(i)\ar[d]&\\
&\CC(i-1)\ar@{-->}[d]\ar[r]^{\alpha_{i}\ \ }&\Sigma ^{i+1}\PP (H_i)\\
\CC\ar[dr]\ar[ur]\ar[uur]\ar[r]&\CC(1)\ar[d]\ar[r]^{\alpha_2}&\Sigma ^3\PP (H_2)\\
&\CC(0)\ar[r]^{\alpha_1}&\Sigma ^2\PP (H_1)}$$ such that
$\CC(i)=\Sigma ^{-1} \CC(\alpha_i)$, where $\CC(\alpha_i)$ denotes
the algebraic mapping cone of $\alpha_i$, and $\PP(H_i)$ denotes a projective resolution of the homology module $H_i$. 

Recall that the algebraic
mapping cone of a chain map $f \colon \CC \to \DD$ is defined as the
chain complex $\CC(f)=\DD \oplus \Sigma \CC $ with boundary map
given by $\bd (x,y)=(\bd x +f(x), \bd y )$. Note that $\Sigma ^n$ is
the shift operator for chain complexes which is defined by $(\Sigma
^n \CC)_i= C_{i-n}$ for every integer $n$.

The algebraic Postnikov system has similar properties to the
Postnikov system in homotopy theory. The maps $\alpha_i\colon
\CC(i-1)\rightarrow \Sigma ^{i+1}\PP(H_i)$ are called $k$-\emph{invariants}
and they are well defined up to chain homotopy equivalence. We can
consider the $k$-invariants as classes in
$\Ext_{\RG_G}^{i+1}(\CC(i-1),H_i)$, since there is an isomorphism
$$[\CC(i-1),\Sigma^{i+1}\PP(H_i)]\cong\Ext_{\RG_G}^{i+1}(\CC(i-1),H_i)$$
between chain homotopy classes of chain maps and the $\Ext$-groups
of chain complexes (see Dold \cite{dold1} for details). The
$k$-invariants $\alpha _i \in \Ext_{\RG_G}^{i+1}(\CC(i-1),H_i)$ are
defined inductively and they uniquely specify $\CC$ up to chain
homotopy equivalence.

 We also need a lifting result for $\RG_H$-modules.
\begin{lemma}\label{lem: lifting}
Let $G$ be a finite group, and $\cF$ be a family of $p$-subgroups in
$G$. Suppose $H\leq G$ controls $p$-fusion in $G$. Then the restriction map  $M \mapsto \Res_H^G(M)$ gives a bijection between the isomorphism classes of conjugation invariant right $\RG_G$-modules and conjugation invariant right $\RG_H$-modules.
\end{lemma}
\begin{proof} A conjugation invariant right $\RG_G$-module $M$ is a functor $\OrG \to \RMod$ which factors through the quotient category
$\OrG \to \SubG$. Here $\SubG$ has objects $K \in \cF$ and morphisms $\Mor_{\SubG}(K,L) = \Mor_{\OrG}(G/K, G/L)/C_G(K)$, where an element $c\in C_G(K)$ acts on a $G$-map defined by $f(eK) =gL$ via the composition  $eK \mapsto cgL$ (see \cite[p.~206]{lueck6}). 

Consider the functor $F\colon \OrH \to \OrG$ given on objects by
$H/K \mapsto G/K$ (see Section \ref{sect:Rim's thm}), and on morphisms by induced maps. First note that
every object of $\SubG$ is isomorphic to an object of $\SubH$, since every $p$-subgroup of $G$ is conjugate to a subgroup of $H$. 
In addition, $ F$ induces a bijection of morphism sets
$$\Mor_{\SubH}(K,L) \to \Mor_{\SubG}(K,L)$$
 since $H$ controls $p$-fusion in $G$.  Suppose that $F(f_1) \approx F(f_2)$, where $f_1(eK) = h_1L$ and $f_2(eK) = h_2L$, for some $h_1, h_2 \in H$. By assumption, there exists $c \in C_G(K)$ such that $ch_2L = h_1L$, or $h_1^{-1}ch_2 \in L\leq H$. But this implies $c \in C_H(K)$ so $f_1 \approx f_2$ and $F$ is injective on morphisms. Given $f\colon G/K \to G/L$ with $K\leq H$,  $f(eK) = gL$ and $g^{-1}Kg \subseteq L\leq H$, we have $g=ch$ for some $c\in C_G(K)$ and $h \in H$, because $H$ controls $p$-fusion in $G$. Hence $f \approx F(f_1)$, where $f_1(eK) = hL$ and $F$ is surjective on morphisms.

Therefore the  functor $F\colon \OrH \to \OrG$ induces an equivalence of categories $$\bar F\colon \SubH \approx \SubG$$ by
\cite[IV.4, Theorem 1, p.~91]{maclane2}. 
\end{proof}

\begin{proof}[Proof of Proposition \ref{prop:mislin_chain}] 
Part (ii) follows from Proposition \ref{prop: finite projective dimension}, so it is enough to prove the existence of $\CC^G$.
By Lemma \ref{lem: lifting}, for each $i \geqslant 0$ there exists a conjugation invariant right $\RG_G$-module $H_i^G$ such that
$\Res_H^G(H_i^G) = H_i(\CC ^H)$.

Consider the Postnikov tower for $\CC^H$. Since $\CC^H(0)=\PP
(H_0(\CC^H))$ there is a complex $\CC^G(0)$ such that $\res_H^G
\CC^G(0) \simeq \CC^H(0)$. In this case, the complex $\CC ^G (0) $
can be taken as a projective resolution of $H_0 ^G $.  Now, we will
show that such a lifting exists for $\CC ^H (i)$ for all $i$. For
this we prove a slightly stronger statement so that we can
carry out an induction. We claim that the following holds for all
$n\geqslant 0$.
\begin{enumerate}
\item $\CC^H(n)$ lifts to a chain complex $\CC^G(n)$
\item The restriction map
$$\res ^G _H\colon
\Ext^*_{\RG_G}(\CC^G(n),N)\rightarrow \Ext^* _{\RG_H}(\CC^H(n), \res
^G _H N)$$ is an isomorphism for all $*\geqslant 0$ and for every
$\RG _G$-module $N$ which is conjugation invariant.
\end{enumerate}
We have already shown that $\CC^H(0)$ lifts to $\CC^G(0)$.  For the
second property, first observe that that $\CC^G(0)$ is chain
homotopy equivalent to a chain complex with single module $H_0^G$
and similarly, $\CC^H(0)\simeq H_0^H$. So, we need to show that
$$\res ^G _H \colon \Ext^* _{\RG_G}(H_0 ^G ,H_1^G)\rightarrow
\Ext^* _{\RG_H}(H_0 ^H ,H_1^H)$$ is an isomorphism. This follows
from  Theorem \ref{thm:filtered mislin_orbit}, because of our assumption on
homology groups.

Now, assume that both (i) and (ii) hold for $n=i-1$.  Then, take
$$\alpha_i ^G \in\Ext^{i+1}_{\RG_G}(\CC^G(i-1),H_i^G)$$ which corresponds
to the class $\alpha_i ^H \in\Ext^{i+1}_{\RG_H}(\CC^H(i-1),H_i^H)$
under the isomorphism given in (ii).  Let $\CC^G(i)= \Sigma
^{-1}\CC(\alpha_i ^G)$.  Then, we have a short exact sequence of
chain complexes
$$\xymatrix{
0\ar[r]&\CC^G(i)\ar[r]&\CC^G(i-1)\ar[r]^{\alpha_i ^G}&\Sigma
^{i+1}\PP (H_i ^G) \ar[r]&0}$$ Since $\res_H^G\alpha_i ^G=\alpha_i
^H$, we have $\res_H^G \CC^G(i)\simeq \CC^H(i)$.  Now, we will show
that (ii) holds for $\CC^G(i)$.  By the Five Lemma, it is enough to show
that
$$\res ^G _H \colon \Ext^*_{\RG _G }(\Sigma^{i+1} \PP (H_i ^G), N) \rightarrow
\Ext^*_{\RG _H}(\Sigma ^{i+1} \PP (H_i ^H),N)$$ is an isomorphism
for all $*\geqslant 0$, and for every $\RG_G$-module $N$ which is
conjugation invariant. But, this follows from Theorem
\ref{thm:filtered mislin_orbit}. 
\end{proof}

Now, we prove one of the main results of this section which allows
us to glue $p$-local chain complexes. We first give a definition.

\begin{definition} Let $\CC$ be a  chain complex over $\RG_G$. We say that $\CC$ has  \emph{homology gaps of length $n$}, if
$H_{i+k}(\CC)= 0$ for  $0 <k< n$, whenever $H_i(\CC) \neq 0$.
\end{definition}

\begin{theorem}
\label{thm:glueing} Let $G$ be a finite group of order $m$. For each prime $p\mid m$, let $\CC^{(p)}$ be a positive
projective chain complex of $\bZ_{(p)}\G_G$-modules. 
Suppose that
\begin{enumerate}\addtolength{\itemsep}{0.2\baselineskip}
\item  $\CC ^{(p)}$ has  homology gaps of length
 $> l(\G_G)$, for all $p\mid m$. 
 \item there exists a graded $\ZG_G$-module
$\HH$ such that $H_i (\CC ^{(p)} )\cong \bZp \otimes _{\bZ} \HH_i $
for all $i\geq 0$, and for all $p \mid m$.
\end{enumerate}
 Then, there is a projective chain complex $\CC$ of
$\ZG_G$-modules  such that $\bZp \otimes _{\bZ} \CC \simeq
\CC^{(p)}$, for each prime $p\mid m$, and  $H_i(\CC) = \HH_i$ for $i \geq 0$.
\end{theorem}

\begin{proof} We will construct $\CC$ inductively. The case $i=0$
is trivial, because in this case we can take $\CC(0)=\PP(\HH_0)$.
Assume now that  $\CC(i-1)$ has been constructed in such a way that
$\bZp \otimes _{\bZ} \CC(i-1)\simeq \CC_{i-1}^{(p)}$ for all $p \mid m$.
If $\HH_i=0$, then we can take $\CC (i)=\CC (i-1)$ and it will satisfy
the condition that $\bZp \otimes _{\bZ} \CC(i)\simeq \CC_i^{(p)}$.
So, assume $\HH_i$ is nonzero. If $i+1 >\bigl ( l(\G_G)+\hdim \CC (i-1)
\bigr )$, then we have
\begin{eqnarray*}
\Ext^{i+1}_{\ZG_G}(\CC(i-1),\PP (\HH_i))&\cong& \bigoplus_{p \mid m}
\Ext^{i+1}_{\bZp\G_G}(\bZp \otimes _{\bZ} \CC(i-1), \HH_i ^{(p)} )\\
&\cong& \bigoplus_{p \mid m}
\Ext^{i+1}_{\bZp\G_G}(\CC^{(p)}(i-1),\HH_i^{(p)})
\end{eqnarray*}
where $\HH_i ^{(p)}=\bZp \otimes _{\bZ} \HH_i$. Note that the above
condition on $(i+1)$ is satisfied since the distance between nonzero
homology groups of $\CC ^{(p)}$ is bigger than $l(\G_G)$. Choose
$\alpha _i \in\Ext^{i+1}_{\ZG_G}(\CC(i-1),\PP (\HH_i))$ so that under
the $p$-localization map, $\alpha _i$ is mapped to the $i$-th
$k$-invariant $\alpha_i ^{(p)}$ of the $p$-local complex $\CC
^{(p)}$, for every $p \mid m$. Let $\CC(i)=\Sigma
^{-1}\CC(\alpha_i)$. For each prime $p \mid m$, we have a diagram
of the form
$$
\xymatrix{0 \ar[r]&\CC(i)\ar[r]&\CC(i-1)\ar[d]^{\varphi _p}
\ar[r]^{\alpha_i}&\PP (\HH_i)\ar [d]\ar[r] & 0 \\ 0 \ar[r] &
\CC^{(p)}(i)\ar[r]&\CC^{(p)}(i-1)\ar[r]^{\alpha^{(p)}_i}&\PP (\HH_i^
{(p)}) \ar[r] & 0 }
$$
where the vertical map $\varphi _p$ is given by the composition
$$\varphi _p \colon \CC (i-1) \to \bZp \otimes _{\bZ} \CC (i-1) \cong \CC
^{(p)}(i-1).$$ The first map in the above composition is induced by
the usual inclusion of integers into $p$-local integers. From this
diagram, it is clear that there is a map $\CC(i)\to \CC^{(p)}(i)$
which induces an isomorphism on homology when it is localized at
$p$. Thus, it gives a chain homotopy equivalence $\bZp \otimes
_{\bZ} \CC(i) \simeq \CC^{(p)}(i)$, for $p\mid m$. This completes the proof.
\end{proof}


We conclude this section with a technique (used in the proof of Theorem A)  for modifying the homology of a given
(finite, projective) chain complex $\CC$ over the orbit category. A projective resolution $\PP\to M$ has \emph{length} $\leqslant \ell$, provided that $P_i = 0$ for $i > \ell$. 

\begin{proposition}\label{prop:adding-subtracting} Let $\G$ be an EI-category.
Let $\varphi\colon \HH_k \to \HH'_k$ be an $\RG$-module homomorphism,
where $\HH_k = H_k(\CC)$. Suppose that both kernel and cokernel of
$\varphi$ admit finite projective resolutions of length $\leqslant \ell$,
and that $H_{k + j} =0$ for $1\leqslant j < \ell$. Then there is a
$\RG$-chain complex $\CC'$ such that $H_i(\CC') = H_i(\CC)$, for $i
\neq k$, and $H_k(\CC') = \HH'_k$.
\end{proposition}

\begin{proof}
First suppose that $\varphi$ is surjective. Let
$$0 \to P_{k+\ell} \to \dots \to P_k \to \ker\varphi \to 0$$
be a projective resolution for $\ker\varphi$. Since $\CC$ is exact
in the range $[k+1, k+\ell)$, we have a chain map
$$\xymatrix{\cdots \ar[r]& 0 \ar[r]\ar[d]& P_{k+\ell} \ar[r]\ar[d]^{f_{k+\ell}}& \dots
\ar[r]& P_{k+1} \ar[r]\ar[d]^{f_{k+1}}& P_k \ar[r]\ar[d]^{f_k}&
\ker\varphi
\ar[r]\ar@{^{(}->}[d]& 0\\
\cdots \ar[r] & C_{k+\ell+1} \ar[r]& C_{k+\ell} \ar[r]& \dots
\ar[r]& C_{k+1} \ar[r]& Z_k \ar[r]& \HH_k\ar[r]& 0 }$$ This gives a
chain map $f\colon \PP \to \CC$, where $f_k \colon P_k \to C_k$ is
the composition of $f_k$ with the inclusion $Z_k \subset C_k$. Let
$\CC'= \CC(f)$ denote the mapping cone of $f$. The induced map
$$\ker\varphi=H_k (\PP) \to H_k(\CC)=\HH_k$$ on homology is given by
the inclusion, and hence $H_k(\CC') = \HH_k'$, with $H_i(\CC') =
H_i(\CC)$ for $i \neq k$. 

Now suppose that $\varphi$ is an injective map, so that
\eqncount
\begin{equation}\label{firstexactsequence}
0 \to \HH_k \xrightarrow{\varphi} \HH'_k \to \coker\varphi \to 0
\end{equation}
is exact. Let $\epsilon\colon\PP \to \coker\varphi$ be a projective resolution of
$\coker\varphi$ of length $\leqslant\ell$, indexed so that $\epsilon\colon
P_k \to \coker\varphi \to 0$. We form the
pull-back
$$\xymatrix{0 \ar[r]& \HH_k \ar[r]\ar@{=}[d]& \widehat \HH_k \ar[r]\ar[d]^{\widehat\varphi}& P_k \ar[r]\ar[d]^{\epsilon}& 0\cr
0 \ar[r]& \HH_k \ar[r]& \HH'_k \ar[r]& \coker\varphi \ar[r]& 0\cr
}$$
of the sequence (\ref{firstexactsequence}) by $\epsilon$, and note that $\widehat \HH_k \cong \HH_k \oplus P_k$.
The chain complex
$$ \dots \to C_{k+1} \to C_k \oplus P_k \to C_{k-1} \to \dots \to C_0 \to 0$$
has homology $\widehat \HH_k$ at $i=k$, and $\widehat\varphi\colon \widehat \HH_k \to \HH'_k$ is surjective. By the pull-back diagram, 
$$\ker \widehat\varphi \cong \ker (\epsilon\colon P_k \to \coker\varphi).$$
 Since $\coker\varphi$ has a projective resolution of length $\leq \ell$, it follows that $\ker \widehat\varphi$ has a projective resolution of length $< \ell$. Hence the assumptions needed for the surjective case hold for  $\widehat\varphi\colon \widehat \HH_k \to \HH'_k$, and we are done by the argument above.

 The general case is done by expressing the map
$\varphi\colon \HH_k \to \HH'_k$ as the composition of a surjection and
an injection.
\end{proof}


\section{The finiteness obstruction}
\label{sect:K-theory}

Let $G$ be a finite group and $\cF$ be a family of subgroups of $G$.
The main result of this section is Theorem \ref{thm:K-theory}: 
given a finite
projective chain complex $\CC$ of $\ZG_G $-modules,  for $\G_G=\OrG$, we can obtain a
finite free complex by taking join tensor of $\CC$ with itself sufficiently
many times.  This result is an adaptation to the orbit category of the fundamental work of Swan \cite{swan1}. We first introduce some definitions, based on the material in L\"uck \cite[\S 10-11]{lueck3}).

Let $\G$ be an EI-category. We denote by $K_0(\ZG)$ the
Grothendieck ring of isomorphism classes of projective $\bZ
\G$-modules and $K_0(\ZG, \free)$  denote the Grothendieck ring
of isomorphism classes of free $\ZG$-modules (under direct sum $M\oplus N$ and 
tensor
product $M\otimes_{\bZ} N$).  
We have an exact sequence of abelian
groups $$0\rightarrow  K_0(\ZG,\free)\rightarrow K_0(\bZ
\G)\xrightarrow{q}\widetilde {K}_0(\ZG)\rightarrow  0$$
defining the quotient group $\widetilde{K}_0(\ZG)$. 

 Note that
$K_0(\ZG, \free)$ is a subring, but not an ideal in general. This
is because the tensor product of a free module with a projective module
is not free in $\ZG$.  For example, if $P $ is a projective
module which is not free, then $P\otimes\un{\bZ } \cong P$ is not a
free $\ZG_G$-module although $\un{\bZ}$ is free when $G \in \cF$.

Given a finite projective chain complex of $\ZG$-modules
$$\CC \colon \,\, 0\rightarrow  C_n\rightarrow  C_{n-1}
\rightarrow
\cdots\rightarrow C_1 \rightarrow C_0\rightarrow  0 $$ we
define
$$\sigma (\CC)=\sum_{i=0}^n (-1)^i[C_i]\in K_0(\ZG)$$ and
$$\widetilde{\sigma}(\CC)=q(\sigma(\CC))\in \widetilde{K}_0(\ZG).$$
The class $\widetilde{\sigma}(\CC)$ is called the \emph{finiteness
obstruction} since it is the only obstruction for $\CC$ to be chain
homotopy equivalent to a finite free chain complex. 

From now on, we
assume that all the chain complexes are positive and  projective. As always, we
assume all modules are finitely generated.

The following are
standard results which show that $\widetilde{\sigma}(\CC)$ is an
invariant, and that it is an obstruction for finiteness. 

\begin{lemma}
If $\CC$ and $\DD$ are chain homotopy equivalent, then $\sigma
(\CC)= \sigma (\DD)$.
\end{lemma}
\begin{proof}
See \cite[11.2]{lueck3}.
\end{proof}

\begin{lemma}
Let $\CC$ and $\DD$ be finite chain complexes of projective $\ZG$-
modules.  Then,
$\sigma (\CC \otimes_{\bZ} \DD )=\sigma (\CC)\cdot\sigma (\DD)$.
\end{lemma}
\begin{proof}
See \cite[11.18]{lueck3} and the sharper result in \cite[11.24]{lueck3}.
\end{proof}

\begin{lemma}
Let $\CC$ be a finite chain complex with $\widetilde{\sigma
}(\CC)=0$. Then $\CC$ is chain homotopy equivalent to a finite chain
complex of free $\ZG $-modules.
\end{lemma}
\begin{proof}
See Swan \cite[Proposition 5.1]{swan1}.
\end{proof}

Given two chain complexes of $\RG$-modules $\CC$ and $\DD$,
consider the
corresponding augmented complexes
$$\widetilde {\CC}\colon\  \cdots\rightarrow C_2 \to  C_1\rightarrow
C_0\rightarrow \un{R} \rightarrow  0$$
$$\widetilde {\DD}\colon \cdots\rightarrow  D_2 \to D_1\rightarrow
D_0\rightarrow \un{R} \rightarrow  0$$
Taking their tensor product, we obtain a complex of the form
$$\widetilde {\CC}\otimes_{R}\widetilde {\DD}\colon \cdots\rightarrow
C_1\oplus D_1\oplus C_0\otimes D_0\rightarrow  C_0\oplus D_0
\rightarrow \un{R} \rightarrow  0.$$ 
\begin{definition}
We define the {\bf join
tensor}, denoted $\CC \divideontimes  \DD$, of two positive augmented chain complexes 
$C$ and $D$ by the formula
$$ \widetilde {\CC \join   \DD}=\Sigma \left ( \widetilde {\CC}\otimes_{R}
\widetilde {\DD} \right ),$$
 where $\Sigma$ denote the suspension of a
chain complex defined by $(\Sigma \CC )_i = C_{i-1} $ for all $i$.
\end{definition}

\begin{lemma}\label{lem:jointensorformula}
Let $\CC$ and $\DD$ be finite chain complexes of projective $\ZG$-modules. Then, $\sigma (\CC \join   \DD)=\sigma (\CC)+\sigma
(\DD)-\sigma (\CC)\cdot \sigma(\DD )$.
\end{lemma}
\begin{proof}
Note that
$(\CC \divideontimes  \DD)_k = C_k \oplus D_k \oplus \bigoplus_{i + j =k-1}  C_i \otimes_{\bZ} D_j $,
for each $k \geq 0$.
Therefore,
$$ \sigma (\CC\join 
\DD)=\sum_k (-1)^k[C_k]+\sum_k(-1)^k[D_k]-\sum_{i+j=k-1}(-1) ^k
[C_i\otimes D_j]$$ 
and the result follows.
\end{proof}

We often express the above formula by writing
$$(1-\sigma (\CC \join   \DD))=(1-\sigma (\CC))(1-\sigma (\DD)).$$
Whenever it is written in this way, one should understand it as a formal
expression of the formula given in Lemma \ref{lem:jointensorformula}.
The main theorem of this section is the following:

\begin{theorem}\label{thm:K-theory}
Let $\G_G =\OrG$ where $G$ is a finite group and $\cF$ is a family of
subgroups in $G$. Given a finite chain complex $\CC$ of projective
$\ZG _G$-modules, there exists an integer $n$ such that $n$-fold join
tensor $\join_n \CC$ of the complex $\CC$ is chain equivalent to a finite complex
of free $\ZG_G$-modules.
\end{theorem}

We need to show that the finiteness obstruction $\widetilde{\sigma }
(\join_n \CC)$ vanishes for some $n$. In the proof we will use a
result by Oliver and Segev \cite{oliver-segev1}.

\begin{proposition}\label{prop:OliverSegev}
Let $G$ be a finite group and let $P$ and $P'$ be any two finitely
generated
projective $\bZ G$-modules. Then, $P \otimes _{\bZ} P'$ is stably
free as a $\bZ G$-module.
\end{proposition}

\begin{proof} See \cite[Proposition C.3]{oliver-segev1}.
\end{proof}

We also need the following splitting theorem for $K_0 (\ZG )$.

\begin{theorem}\label{thm:splittingK-theory}
Let $\G$ be a $EI$-category.  Then, the map
$$ K_0(S) : K _0 (\ZG ) \to \bigoplus _{x \in \Is (\G) }  K_0(\bZ
[x] ),$$ defined by $[P]\to [S_x(P)]$ on each $x \in \Is (\G)$, is
an isomorphism. The same holds when $K_0$ is replaced by
$\widetilde{K}_0$.
\end{theorem}
\begin{proof}
See L\"uck \cite[Proposition 11.29]{tomDieck2}.
\end{proof}
As a consequence of this theorem, if $\G$ is finite then $\widetilde{K} _0 (\ZG)$
 is  finite: in this case $\G$ has finitely many isomorphism classes of objects $x \in \Ob(\G)$, and $\Aut[x]$ is a finite group (apply
Swan \cite[Prop.~9.1]{swan1a}).  In particular, if
$\G =\OrG$, then the group $\widetilde{K} _0 (\G )$ is finite.

From now on we assume $\G_G =\OrG$ for some finite group $G$, relative
to some family $\cF$. The splitting theorem above can also be used
to give a filtration of $\widetilde{K} _0 (\G_G)$. Recall that every
projective $\ZG_G$-module is of the form
$$ P\cong\bigoplus_{H \in T} E_H S_H P$$ where $T$ is a set of
representatives of conjugacy classes of elements in $\cF$. So,
another way to express the above splitting theorem is to write
$$ K_0(\ZG_G)\cong \bigoplus_{H \in T} K_0(\ZG_G)_H$$ where
$K_0(\ZG_G)_H=\{[P]\mid E_H S_H P\cong P\}$. Note that this is only
a splitting as abelian groups, but using this we can give a
filtration for the ring structure of $K_0 (\ZG_G )$. Let
$$ \varnothing=T_0\subseteq  T_1\subseteq  \cdots\subseteq  T_m=T$$
be a filtration of $T$ such that if $H\in T_i$ and $K\in T_j$ and
$\leftexp{g} H\leq K$ for some $g \in G$, then $i\leqslant j$.
This gives a filtration
$$ 0=K_0(\ZG_G )_0\subseteq  K_0(\ZG_G )_1\subseteq  \cdots\subseteq
K_0(\ZG_G)_m=K_0(\ZG_G )$$ where
$$K_0( \ZG_G)_i=\{[P]\vv P=\bigoplus_{H \in T_i} E_H S_H P\}.$$

\begin{lemma}\label{lem:decomposition}
Let $V$ be a $\bZ [N_G (H) / H ]$-module and $W$ be a
$\bZ [ N_G (K)/K]$-module. Then,
$$ E_{H}V\otimes_{\bZ} E_{K} W \cong\bigoplus_{HgK \in H\backslash G/
K}E_{{H\cap
\leftexp{g}{K}}}(\Res _{ H\cap\leftexp{g}{K}}  E_{H} V \otimes _{\bZ}
\Res _{H\cap\leftexp{g}{K}}
E_{K}W).$$
\end{lemma}
\begin{proof} Applying the definition, we get
\begin{eqnarray*} E_{H}V\otimes_{\bZ} E_{K} W
&=&(V\otimes W)\otimes_{\bZ [\Aut(G/H)\times \Aut(G/K)]} \bZ \Map_G
(?,G/H \times G/K)
\end{eqnarray*}
where $\Map_G (X,Y)$ denotes the set $G$-sets from $X$ to $Y$ (see
\cite[11.30]{tomDieck2} for a similar computation). Since
$$G/H\times G/K=\coprod_{HgK\in H\backslash G/K}G/{(H\cap\leftexp{g}
{K})},$$
the module $E_{H}V\otimes_RE_{K}W$ decomposes as
$$\bigoplus_{HgK \in H\backslash
G/K}E_{{H\cap\leftexp{g}{K}}}U_{H\cap\leftexp{g}{K}}$$ where
$U_{H\cap\leftexp{g}{K}}$ are  $N_G
(H\cap\leftexp{g}{K})/(H\cap\leftexp{g}{K})$-modules. Applying
$S_{{H\cap\leftexp{g}{K}}}$, we find
\begin{eqnarray*}
U_{H\cap\leftexp{g}{K}}= S_{ {H\cap\leftexp{g}{K}} } ( E_{H}V\otimes_
{\bZ} E_{K} W )
&=&\Res_{{H\cap\leftexp{g}{K}}}(E_{H}V\otimes_{\bZ} E_{K} W)\\
&=&\Res _{H\cap\leftexp{g}{K}} E_{H}V\otimes_{\bZ} \Res
_{H\cap\leftexp{g} {K}} E_K W.
\end{eqnarray*}
\end{proof}
\begin{lemma}
$K_0(\ZG_G )_i$ is an ideal of $K_0(\ZG_G )$ .
\end{lemma}
\begin{proof}
For $E_HS_H P$ and $E_KS_K Q$, we have
$$E_H S_H P\otimes_{\bZ} E_K S_K Q=\bigoplus_L E_L V_L $$ where $L=H\cap
\leftexp{g}{K}$ for some $g \in G$.  So, if $H\in T_i$, $K\in T_j$, and
$L\in T_k$, then $k\leqslant i, j$.
\end{proof}

Now, Theorem \ref{thm:K-theory} follows by induction from the
following proposition.

\begin{proposition}
Let $\CC$ be a finite chain complex of projective $\ZG_G $-modules.
If $\widetilde{\sigma }(S_H \CC)=0$ for all $H \in T\smallsetminus
T_i$, then there is an $n$ such that $\widetilde{\sigma }(S_H (\join 
_n \CC ))=0$ for all $H \in T\smallsetminus T_{i-1}$.
\end{proposition}

\begin{proof}  An element in $\sigma (\CC)$ can be expressed as a sum
$u+\sum_j v_j+w$ where
$$ u=\sum_{H \in T_{i-1}}\sigma (E_HS_H \CC),\quad \sum_j v_j = \sum_{H \in
T_i\smallsetminus T_{i-1}}\sigma (E_HS_H \CC ), \quad w = \sum_{H \in
T\smallsetminus T_i}\sigma (E_HS_H \CC ). $$
 By Lemma
\ref{lem:jointensorformula}, we have
$$ 1-\sigma (\join   _n \CC )=(1-\sigma (\CC ))^n=(1-(u+\sum_j v_j+w))^n
\in K_0(\ZG_G )$$ 
So, $$ 1-\sigma (\join   _n \CC )\equiv \Bigl
(1-(\sum_jv_j+w)\Bigr )^n \quad {\rm mod} \, K_0(\ZG_G )_{i-1}$$ By
Lemma \ref{lem:decomposition}, it is easy to see that
$v_j \cdot v_k \equiv 0 {\rm \ \big (mod}\, K_0(\ZG_G )_{i-1}\big )$, for $j \neq k$.
Note that
$$v_j\cdot v_j \equiv \text{stably\ free\, }  \big (\text{mod\ } 
K_0(\ZG_G )_{i-1}\big )$$ 
by Proposition \ref{prop:OliverSegev}.
To complete the proof, observe that modulo $K_0(\ZG_G )_{i-1}$,
\begin{eqnarray*}
1-\sigma (\join   _n \CC ) &\equiv & \Bigl (1-(\sum_j v_j+w)\Bigr )^n\\
&\equiv &1+\sum_{k=1}^n\binom{n}{k}(-1)^k(\sum_jv_j+w)^k\\
&=&1+\sum_{k=1}^n\binom{n}{k}(-1)^k\sum_j k v_jw^{k-1}+\,\rm{stably}\,
\rm{free}\\
&=&1+n \sum_{k=1}^n\sum_j\binom{n-1}{k-1}(-1)^kv_jw^{k-1}+\,\rm
{stably}\,\rm{free}.\\
\end{eqnarray*}
This shows that $\sigma (\join   _n \CC ) $ is stably free for some
$n$, since $\widetilde {K}_0(\ZG_G )$ is a finite group.
\end{proof}

\section{Realization of free chain complexes}
\label{sect:realization}

Let $X$ be $G$-CW complex, and let $\cF$ be a family of
subgroups of $G$. 
Throughout this section, $R$ denotes a
commutative ring and $\G_G$ denotes the orbit category $\OrG$.
\begin{definition}
We say that a  $G$-CW complex $X$ has \emph{isotropy in $\cF$}, provided that $X^H \neq \emptyset$ implies $H \in \cF$, for all $H\leq G$. 
\end{definition}
The main result of this section is Theorem \ref{thm:realization}, which shows that under certain conditions a finite free chain complex over the orbit category can be realized by a finite $G$-CW complex with isotropy in $\cF$. This is a generalization of Swan \cite[Theorem A]{swan1}, which is based on a construction of Milnor \cite[3.1]{swan1}. 

\smallskip
Associated to a $G$-CW complex  $X$ with isotropy in $\cF$, there is a chain complex of $\RG_G$-modules
defined by
$$
\uC{X}\colon  \quad \cdots \xrightarrow{\bd_{n+1}} \uR{X_n}
\xrightarrow{\bd_{n}}  \uR{X_{n-1}} \rightarrow
\cdots\xrightarrow{\bd_{1}}  \uR{X_0} \to 0$$ where $X_i$ denotes
the set of $i$-dimensional cells in $X$ and $\uR{X_i}$ is the
coefficient system with $ \uR{X_i}(H)= R[X_i^H]$. We denote the
homology of this complex by $\uH{X}$, and in particular
$$\uH{X}(H)=H_*(X^H;R).$$

 Given a chain complex
$\CC$ of $\RG_G$-modules,  there is a \emph{dimension function}
$\Dim \CC\colon \cF \to \bZ$, constant on conjugacy classes of subgroups, defined by
$$(\Dim \CC )(H)=\dim \CC (H),$$
 for all $H \in \cF$, where the
dimension of a chain complex of $R$-modules is defined in the usual
way as the largest integer $d$ such $C_d\neq 0$.

It will be convenient to write
$(H) \leq (K)$ whenever $H^g \leq K$ for some $g\in G$. Here $(H)$ denotes the set of subgroups conjugate to $H$ in $G$.

\begin{definition}\label{def: monotone}
We call a function $\un{d}\colon\cF\to  \bZ$ \emph{monotone} if it satisfies the
property that $\un{d}(K) \leqslant \un{d}(H)$ whenever $(H) \leq (K)$. We say that a monotone
 function  $\un{d}$ is \emph{strictly monotone} if
$\un{d}(K) < \un{d}(H)$, whenever
 $(H) \leq (K)$ and $(H) \neq (K)$. \qed
\end{definition}

Note that  $\un{d}$ monotone implies that $\un{d}$ is constant on conjugacy classes (such functions are usually called  \emph{super class functions}).  We remark that the dimension function of a projective chain complex is always
monotone: if $(E_H P) (K)\neq 0$, then $(E_H P) (L)\neq 0$ for
every $L \leq K$.

A chain complex $\CC$ of $\RG_G$-modules is \emph{connected} if $\CC$ is positive and 
$H_0(\CC) = \un{R}$.

\begin{definition} Let $\un{n}\colon \cF \to \bZ$ be a monotone, non-negative function.  A complex $\CC$ of $\RG_G$-modules is called an $\un{n}$-\emph{Moore complex}
if it is connected,  and
for all $H \in \cF$, the reduced homology  $\widetilde{H}_i (\CC(H))  = 0$,
for $i \neq \un{n}(H)$.
 \qed
\end{definition}

A special case of an $\un{n}$-Moore complex is a homology $\un{n}$-sphere.

\begin{definition} We say that a complex $\CC$ of $\RG_G$-modules is an
 \emph{$R$-homology
$\un{n}$-sphere} if it is an $\un{n}$-Moore complex, and for all $H\in
\cF$, we have $\widetilde{H}_i (\CC(H)) \cong R$,  for $i=\un{n}(H)$.
A homology $\un{n}$-sphere is called \emph{oriented} if the $N_G(H)/H$-action
is trivial on the homology of $\CC(H)$ for all $H \in \cF$.
\end{definition}

The chain complex associated to the unit sphere $X=S(V)$ of a real or complex
representation $V$ of $G$ is an example of a $\bZ$-homology $\un{n}$-sphere, where $\un{n}(H) = \dim X^H$.
A $G$-CW complex $X$ with this property
is  a \emph{homotopy representation} in the sense of tom Dieck (see
\cite[Chap.~II, Def.~10.1]{tomDieck2}), provided that its dimension
function is strictly monotone.  We will not use this terminology further.

We now introduce a technique to remove free
modules above the homological dimension from a chain complex, without changing its chain
homotopy type. For this delicate process we first need some algebraic
lemmas.

\begin{definition} Let $\G$ be an EI-category.
A free $\RG$-module $F$ is called \emph{isotypic} of type $x\in \Ob ( \G)$ if it is
isomorphic to a direct sum
of copies of the free module $E_xR[x]$.
\end{definition}

For extensions involving isotypic modules we have a splitting property.

\begin{lemma}\label{lem:splitting}
Let $$\cE:   0 \to F \to F' \to M\to 0$$ be a short exact sequence
of $\RG$-modules over an EI-category $\G$,  such that both $F$ and $F'$ are isotypic free modules of
the same type $x\in \Ob (\G)$. If $M(x)$ is $R$-torsion free, then $\cE$ splits
and $M$ is stably free.
\end{lemma}

\begin{proof}
It is enough to prove the result in case $F= E_xR[x]$,
where $x\in \Ob \G$. The general case
follows from this by an easy induction. Consider the extension
$$\cE : 0 \to E_xR[x] \xrightarrow{j} F \to M\to 0\ .$$
 By the adjointness property
$$\Hom_{\RG}(E_xR[x], N) \cong \Hom_{R[x]}(R[x], N(x))$$
for any $\RG$-module $N$. We apply this to the given injection
$j \colon E_xR[x] \to F' = (E_xR[x])^m$. Since
$$\cE(x) \colon 0 \to R[x] \xrightarrow{j} R[x]^m \to M(x)\to 0$$
has $R$-torsion free cokernel $M(x)$, this sequence splits over $R[x]$.
By the naturality of the adjointness property, we get a splitting of $j$
over $\RG$.
\end{proof}
Recall that $\hdim\CC(H)$ denotes the homological dimension  of the chain complex $\CC(H)$.

\begin{proposition}\label{prop:killing}
Let $\CC$ be a finite free chain complex of $\RG_G$-modules, and let $H \in \cF$  have the property that $\hdim\CC(H) < d:=\dim \CC (H)$.
Suppose that $\dim \CC(K) \leqslant (d-2)$ for all $(H) \leq (K)$, $(H)\neq (K)$.
Then $\CC\simeq \DD$, where $\DD$ is a finite free complex with $\dim \DD (H)=d-1$, and $\dim \DD(K) = \dim \CC(K)$ for all $(K)\neq (H)$.
\end{proposition}

\begin{proof}
Consider the subcomplex $\CC'$ of $\CC$ formed
by free summands of $\CC$ isomorphic to $\uZ{G/K}$, with $(G/K)^H\neq 0$ or equivalently $(H) \leq (K)$. The boundary maps of $\CC'$ are the restrictions of the usual boundary
maps to these submodules. Since $\dim \CC(K) \leqslant (d-2)$ for all $(H) \leq (K)$ such that $(H)\neq (K)$,
 the free
modules $C'_d$ and $C'_{d-1}$ are isotypic of type $G/H$.
 We have
$$\CC' : 0\to C' _{d} \to C' _{d-1}\to
 \cdots \to C'_1 \to C'_0 \to 0  $$ where
$d=\dim \CC (H)$. Note that  $\CC (H)=\CC' (H)$, so the map
$\bd _d\colon  C' _{d} \to C' _{d-1}$ is injective by the condition that
$\hdim \CC(H) < \dim \CC (H)$.
Now we can apply
Lemma \ref{lem:splitting} to the extension
$$0 \to C'_d \xrightarrow{\bd_d} C'_{d-1} \to \coker \bd_d \to 0$$
and conclude that $\coker(\bd_d)$ is a stably free $\RG_G$-module.
By adding elementary chain complexes to $\CC$ of the form $\uZ{G/H} \xrightarrow{\id} \uZ{G/H}$ in the adjacent dimensions $(d-1)$ and $(d-2)$, we can assume that $\coker(\bd_d)$ is free.
Consider the diagram
$$\xymatrix{\quad \quad \cdots \ar[r]& 0 \ar[r]\ar[d]& C'_{d}
\ar[r]^{\id}\ar[d]^{\id}& C'_d  \ar[r]\ar[d]^{\bd _d}& 0 \ar[r]\ar[d]
&\cdots \ar[r]& 0 \ar[r]\ar[d]&0\\
\CC' \colon \cdots \ar[r]& 0 \ar[r]\ar[d]& C'_{d} \ar[r]\ar[d]&
C'_{d-1} \ar[r]\ar[d] & C'_{d-2} \ar[r] \ar@{=}[d] & \cdots \ar[r]&
C'_0\ar[r]\ar@{=}[d] & 0\\
\DD ' \colon  \cdots \ar[r] & 0 \ar[r]& 0 \ar[r]& \coker \bd _d
\ar[r] & C'_{d-2} \ar[r] & \dots \ar[r]& C'_0 \ar[r]& 0 }$$
The chain complex $\DD '$ is a chain complex of free modules and it
is chain homotopy equivalent to $\CC '$. Now define $\DD$ as the
push-out in the the following diagram:
$$\xymatrix{ & \ker \ar@{=}[r]\ar[d] & \ker \ar[d] & \\
& \CC ' \ar[r]\ar[d] & \CC \ar[r]\ar[d] & \CC / \CC ' \ar@{=}[d] \\
& \DD ' \ar[r] & \DD \ar[r] & \CC / \CC ' \\}$$ Since, $\CC'$ and
$\DD'$ are chain homotopy equivalent, then $\CC$ and $\DD$ are chain
homotopy equivalent. Also, note that $\dim \DD (H)=\dim \DD '
(H)=(d-1)$, and $\dim \DD(K) = \dim \CC(K)$ for all $(K)\neq (H)$.
\end{proof}

This immediately gives the following.

\begin{corollary}\label{cor:strictlymonotone}
Let $\CC$ be a finite free chain complex of $\RG_G$-modules.
Suppose that $\CC$ is a homology $\un{n}$-sphere, with $\un{n}$ strictly
monotone. Then $\CC$ is chain homotopy
equivalent to a complex $\DD$ with $\Dim \DD
=\un{n}$.
\end{corollary}
\begin{proof} Since $\CC$ is a homology $\un{n}$-sphere, $\un{n}(K) = \hdim\CC(K)$, for all $K \in \cF$. We apply the previous result to a subgroup $H$, which is maximal with respect to the property that $\hdim\CC(H) < d:=\dim \CC (H)$. Then $\un{n}(K) = \dim \CC(K)$ for all $K \in \cF$ larger than $H$.
Since $\un{n}$ is strictly monotone, $\dim \CC(K) \leqslant (d-2)$ for all
$(H) \leq (K)$, $(H)\neq (K)$. This process can be repeated until $\Dim \DD
=\un{n}$.
\end{proof}

When the dimension function of $\CC$ is not strictly monotone, we get a weaker result. Following Section
\ref{sect:definitions},  we define $l(H,K)$ as the maximum length of a chain
of conjugacy classes of subgroups
$$(H) = (H_0)\, \lneqq\, (H_1)\, \lneqq \,\dots\dots \,\lneqq\, (H_l) = (K)$$
where all $H_i \in \cF$, $0 \leqslant i \leqslant l$.

\begin{corollary}\label{cor:atmost2}
Let $\CC$ be a finite free chain complex of $\RG_G$-modules, and let $\un{n}\colon \cF \to \bZ$ be a monotone  function such that
 $\hdim \CC(H) \leqslant \un{n}(H)$ for all $H\in \cF$. Assume that
 $l(H,K) \leqslant k$ whenever $\un{n}(H) = \un{n}(K)$.
Then, $\CC$ is chain homotopy
equivalent to a complex $\DD$ which satisfies
$D _i (H)=0$ for all $i > \un{n} (H) +k$.
\end{corollary}
\begin{proof}
Let $$(H) = (H_0)\, \lneqq\, (H_1)\, \lneqq \,\dots\dots \,\lneqq\, (H_l) = (K)$$ be a maximal length chain of subgroups in $\cF$ with $\un{n}(H) = \un{n}(K)$.
Since $\un{n}$ is monotone, $\un{n}(H_i) = \un{n}(H)$ for $0 \leqslant i \leqslant l$. By repeated application of
Proposition \ref{prop:killing}, working down from the maximal element $K$, we can obtain $\dim \CC(H_{l-i}) = \un{n}(H) + i$, for $0 \leqslant i \leqslant l$. Since
$l= l(H,K) \leqslant k$, we have $\dim \CC(H) \leqslant \un{n}(H) + k$ as required.
\end{proof}

The main purpose of this section is to prove the following theorem:

\begin{theorem}[Pamuk \cite{pamuks1}]
\label{thm:realization} Let $\CC$ be a finite free chain complex of
$\ZG_G$-modules. Suppose $\CC$ is an $\un{n}$-Moore complex such that
$\un{n}(H) \geqslant 3$ for all $H\in \cF$. Suppose further that
$C _i (H)=0$ for all $i > \un{n} (H) +1$, and all $H\in \cF$.
Then there is a finite $G$-CW complex $X$, such that
$\uCZ{X}$ is chain homotopy equivalent to $\CC$, as chain complexes
of $\ZG_G$-modules.
\end{theorem}

Note that the resulting complex $X$ will have isotropy in $\cF$. We first prove a lemma (compare \cite[Thm.~13.19]{lueck3}).

\begin{lemma}
\label{lem:addingcells}
Let $X$ be a finite $G$-CW complex. Suppose that we are given a
 free $\ZG_G$-module $F$, and a $\ZG_G$-module homomorphism $\varphi
\colon F\to H _n(X^{\textbf{?}}; \bZ)$, for some $n \geqslant 2$.  Assume
further that $X^H$ is $(n-1)$-connected
for every $H\in \cF$ such that $\uZ{G/H}$ is a summand of $F$. Then,
by attaching $(n+1)$-cells to $X$, we can obtain a $G$-CW complex $Y$ such
that $$H_i(X ^{\textbf{?}} ;\bZ )\cong H_i(Y^{\textbf{?}} ; \bZ)\
{\rm for} \ i\neq n,n+1,
$$ and
$$0\to H_{n+1} (X ^{\textbf{?}} ; \bZ ) \to
H_{n+1}(Y ^{\textbf{?}} ; \bZ )\to F  \xrightarrow{\varphi} H_{n} (X
^{\textbf{?}} ; \bZ ) \to H_{n}(Y ^{\textbf{?}} ; \bZ )\to 0$$ is
exact.
\end{lemma}

\begin{proof}
Let $Z$ be a wedge of $n$-spheres with a $G$ action on them such
that $\widetilde{H} _n(Z^{\textbf{?}} ; \bZ )\cong F$  as $\bZ
\G_G$-modules. We want to construct a map $f\colon Z \to  X$ realizing
$\varphi $.
 But  $H_n(X^H;\bZ) \cong \pi_n(X^H)$, for every $H\in
\cF$ such that $\uZ{G/H}$ is a summand of $F$, since $X^H$ is assumed
to be $(n-1)$-connected. Therefore, we can represent the images of an
$\bZ[N_G(H)/H]$-basis under $\varphi$  for the isotypic summand in $F$ of
type $G/H$ by maps $f_i\colon S^n \to X^H$. We extend these maps equivariantly
to maps $\bar f_i\colon S^n \times G/H \to X$. By repeating this construction
for each type $G/H$ in $F$, we obtain an
equivariant map $f\colon Z \rightarrow X$ realizing $\varphi $.
Take $Y$ to be the mapping cone of $f$. Then, it is easy to see that
$Y$ satisfies the desired conditions.
\end{proof}

We also need the following lemma:

\begin{lemma}\label{lem:replace}
Let $\CC$ be a finite free chain complex of
$\ZG_G$-modules. Suppose that $\CC$ is connected, and
$H_i(\CC) = 0$, for $ i =1,2$. Then, $\CC$ is chain homotopy equivalent to a complex
of the form
$$ \cdots \to C_n \to C_{n-1}\to \cdots \to C_3 \to C_2 (X) \to C_1
(X) \to C_0 (X) \to 0$$
 where $C_2 (X) \to C_1 (X) \to C_0 (X) \to
0$ is the initial part of the chain complex $\uCZ{X}$, for some  $G$-CW complex $X$ with isotropy in $\cF$, and $X^H$ simply-connected for all $H \in \cF$.
\end{lemma}

\begin{proof} There is a $G$-CW complex $E_{\cF} G$ satisfying the
properties
\begin{enumerate}
\item All isotropy subgroups of $E_{\cF} G$ are in $\cF$,

\item For every $H \in \cF$, the fixed point set $(E_{\cF} G)^H$ is
contractible \cite[Theorem 1.9]{lueck4}.
\end{enumerate}

The chain complex $\DD:=\uCZ{(E_{\cF} G)}$ of this space gives a
free resolution of $\un{\bZ}$ as a $\ZG_G$-module.
Since $H_i(\CC) = 0$, for $ i =1,2$,
 the following sequences are both exact
 \eqncount
\begin{equation}\label{two_sequences}
\vcenter{\xymatrix{0 \ar[r] & A  \ar[r] & C_{2} \ar[r]^{\bd^C_2}&
C_1 \ar[r]^{\bd ^C _1} & C_0\ar[r] & \un{\bZ} \ar[r]\ar@{=}[d] & 0\\
0 \ar[r] & B \ar[r] & D_{2} \ar[r]^{\bd ^D _2} & D_1 \ar[r]^{\bd ^D
_1} & D_0\ar[r] & \un{\bZ} \ar[r] & 0}}
\end{equation}
 where $A =\ker\bd ^C_2 $
and $B=\ker\bd ^D_2$. 

By an \emph{elementary operation} on a sequence
$A \to C_2 \to C_1 \to C_0$ we mean adding or removing trivial free summands $F \xrightarrow{id} F$ in adjacent dimensions. It is clear that elementary operations don't change the chain homotopy type of the upper and lower sequences in diagram (\ref{two_sequences}).

Then, by Schanuel's Lemma \cite[1.1]{swan1}, there exist free
modules $F$ and $F'$ such that $ A\oplus F \cong B\oplus F'$. In fact, the
argument in Schanuel's lemma can be extended to say that the induced
isomorphism $\gamma\colon A\oplus F \cong B\oplus  F'$ comes from a chain isomorphism after
a sequence of elementary operations (compare \cite[p.~279]{lueck3}). 

In other words, there exists 
a chain isomorphism
\eqncount
\begin{equation}\label{two_sequences_stable}
\vcenter{\xymatrix{0 \ar[r] & A \oplus F  \ar[r]\ar[d]^{\cong}_{\gamma} & C_{2} \oplus F_2 \ar[r]
\ar[d]^{f_2}_{\cong}& C_1 \oplus F_1 \ar[r]\ar[d]^{f_1}_{\cong} & C_0 \oplus F_0 \ar[r]
\ar[d]^{f_0}_{\cong} & \un{\bZ} \ar[r]\ar@{=}[d] & 0\\
0 \ar[r] & B \oplus F'\ar[r] & D_2 \oplus F'_2  \ar[r] & D_1 \oplus F'_1 \ar[r] & D_0
\oplus F'_0 \ar[r] & \un{\bZ} \ar[r] & 0}}
\end{equation}
  for some suitable choices of
free modules, where the upper and lower sequences in diagram (\ref{two_sequences_stable}) are obtained from those in diagram (\ref{two_sequences}) by elementary operations  (see Proposition 3.3.3 in \cite{pamuks1}).
 
In the first step, we  stabilize $(A \to C_2)\mapsto (A\oplus F \to C_2\oplus F)$, by adding the identity on $F$, and similarly   $(B \to D_2) \mapsto (B \oplus F' \to D_2\oplus F')$.
We therefore have a chain map
$$\xymatrix{0 \ar[r] & A \oplus F  \ar[r]\ar[d]^{\cong}_{\gamma} & C_{2} \oplus F \ar[r]
\ar[d]& C_1  \ar[r]\ar[d]& C_0 \ar[r]
\ar[d] & \un{\bZ} \ar[r]\ar@{=}[d] & 0\\
0 \ar[r] & B \oplus F'\ar[r] & D_2 \oplus F'  \ar[r] & D_1  \ar[r] & D_0\ar[r]
& \un{\bZ} \ar[r] & 0}$$ which is a chain homotopy equivalence (by composition with the chain map in (\ref{two_sequences_stable})).
After an elementary operation on $\CC$, we can use the isomorphism $\gamma\colon A\oplus F \cong B\oplus  F'$  to splice the bottom
sequence to $\CC$, and obtain a chain homotopy equivalence
$$\xymatrix{\cdots \ar[r]& C_4 \ar[r]\ar@{=}[d]  & C_3 \oplus F  \ar[r]\ar[d] & C_{2} \oplus F \ar[r]
\ar[d]& C_1  \ar[r]\ar[d]& C_0 \ar[r]
\ar[d] & \un{\bZ} \ar[r]\ar@{=}[d] & 0\\
\cdots \ar[r] &C_4 \ar[r] & C_3 \oplus F \ar[r] & D_2 \oplus F' \ar[r] & D_1
\ar[r] & D_0\ar[r] & \un{\bZ} \ar[r] & 0}$$ 
The top sequence is
chain homotopy equivalent to $\CC$, so to complete the proof we need
to show that the sequence $ D_2 \oplus F' \to D_1\to D_0 \to 0$ can be
realized as the first three terms of a chain complex of a $G$-CW complex $X$, with isotropy in $\cF$, such that $X^H$ is simply
connected for all $H \in \cF$:
since $E _{\cF } G$ is contractible, using Lemma
\ref{lem:addingcells}, we can attach free $2$-cells to  its two skeleton
$E_{\cF} G^{(2)} $. The resulting complex $X$ will have the desired
properties.
\end{proof}

Now, we are ready to prove Theorem \ref{thm:realization}.

\begin{proof}[Proof of Theorem \ref{thm:realization}]
We can assume that the complex $\CC$ is of the form given in Lemma
\ref{lem:replace}. We obtain a map  $\varphi\colon C_3 \to
C_2(X^{(2)})$ which induces an isomorphism $Z_2 (\CC) \to Z_2
(X^{(2)})$ between $2$-cycles of these chain complexes. This is the starting point for an inductive argument based on applying Lemma \ref{lem:addingcells} at each step.

Fix $n \geqslant 2$, and assume by induction that
there is an $n$-dimensional $G$-CW complex $X^{(n)}$, and a chain map
$$\xymatrix@C-3pt{\cdots \ar[r] & C_{n+2} \ar[r]\ar[d] & C_{n+1} \ar[r]\ar
[d]& C_{n} \ar[r]\ar[d]& \dots \ar[r]&  C_1 \ar[r]\ar[d]& C_0\ar[r]
\ar[d] & 0\\ \cdots \ar[r] & 0 \ar[r]& Z_n(X^{(n)})\ar[r] & C_{n} (X^ {(n)})
\ar[r]& \dots \ar[r]& C_1 (X^{(n)}) \ar[r]& C_0 (X^{(n)}) \ar[r]& 0
}$$
which induces an homology isomorphism for dimensions less
than or equal to $(n-1)$, and at dimension $n$ the induced map $Z_n
(\CC) \to Z_n (X^{(n)})$ is an isomorphism.

Note that $\dim \CC(H) \leqslant \un{n}(H) + 1$ by assumption.
If $\uZ{G/H}$ is a summand of $C_{n+1}$, then
$(n+1) \leqslant \dim \CC(H) \leqslant \un{n}(H) + 1$ implies $\un{n}(H) \geqslant n$, and hence the $H$-fixed set of $X^{(n)}$ is $(n-1)$-connected.
We can now apply Lemma
\ref{lem:addingcells} to the map
$\varphi\colon C_{n+1} \to H_n (X^ {(n)}; \bZ )$ defined by the composition
$$\phi \colon C_{n+1} \to Z_n (\CC) \cong Z_n (X^{(n)} ) \to  H_n (X^
{(n)} ; \bZ ).$$
Let us call the resulting complex $X^{(n+1)}$. Note
that there is a chain map $\CC \to \CC (X^{(n+1)})$ which induces an
isomorphism on homology for dimensions  $\leqslant n$, and
at dimension $n+1$ we have an isomorphism $Z_{n+1} (\CC) \to Z_{n+1
} (X^{(n+1) }) $. Since $\CC$ is finite dimensional, after finitely
many steps, we will obtain a finite dimensional $G$-CW complex $X$
and a chain map $f\colon \CC \rightarrow \CC(X)$ which induces
isomorphism on homology for all dimensions. Since both $\CC$ and
$\CC(X)$ are free $\ZG_G$-chain complexes, $f$ is a chain homotopy
equivalence as desired.
\end{proof}

\section{The proof of Theorem A}
\label{sect:construction-new}

Let $G = S_5$, the
symmetric group of order 120 permuting $\{1, 2,3,4, 5\}$, and let $S_4\leq G$
denote the permutations fixing $\{5\}$. We work relative to the family $\cF$
of rank 1 subgroups of $2$-power order. Let $\G_G = \OrG$. 
Our family $\cF$ consists of the
subgroups of $G$ which are conjugate to one the subgroups in the set
$$\{ 1, C_2^A, C_2^B, C_4\}$$ 
where
$\cy{2}^A = \la (12)(34)\ra$, $\cy{2}^B = \la (12)\ra$, and $\cy{4}
= \la (1234)\ra$.  In addition we will consider the
Sylow subgroups $\cy{3} = \la (123)\ra$ and $\cy{5} = \la
(12345)\ra$.
It is convenient to note that for $H = S_4 \leq G$, we have
$$N_H(\cy{4}) = D_8=N_H(\cy{2}^A)=N_G(\cy{4}),$$
while $N_H(\cy{2}^B) = E = \la (12), (34)\ra$, and
$N_H(\cy{3}) = S_{\{123\}}$. On the other hand,
$ N_G(\cy{2}^B) = \la (12), S_{\{345\}}\ra$ and
$N_G(\cy{3}) = S_{\{123\}}\times \la (45)\ra$.

Our strategy will be to construct finite projective complexes $\CC^{(p)}$ with isotropy in $\cF$ 
over $ \bZp \G_G$, for each prime $p$ dividing the order of $|G|$, which are $R$-homology $\un{n}$-spheres with respect to the
the same  homology dimension
function $\un{n}$. The
gluing theory of Section \ref{sect:chain complexes}, Theorem
\ref{thm:glueing}, will be used to construct a finite projective
$\bZ$-homology $\un{n}$-sphere over $\ZG_G$ from this data. Then the join
construction from Section \ref{sect:K-theory} will allow us to find
a finite free complex, to which the realization theorem of
Section \ref{sect:realization} will apply.

We introduce the notation $\RR_0$ for the $\RG_G$-module defined by $\RR_0(K) = 0$, for $K \neq 1$, and $\RR_0(1) =
R$ with trivial $G$-action. In other words, $\RR_0 = I_1(R)$ as defined in Section  \ref{sect:definitions}. 

\subsection{The case $p=2$.}
Let $H = S_4 \leq G$, $R = \bZ_{(2)}$ and consider the standard $H$-action on the $2$-sphere given by the rotational symmetries
of the octahedron. Let $X$ denote the $H$-CW complex associated to  
the first barycentric subdivision of the  octahedron. 
Then $X$ has isotropy in the family consisting of the cyclic subgroups of $H$  of orders $\leq 4$.

Let $\G_H = \OrH$ denote the orbit category for $H$ with respect to the family $\cF_H =\cF \cap H$.
Consider the chain complex $\uC{X}$ as a chain complex of $\RG_H$-modules, by restricting this functor to the full subcategory $\G_H$ of the orbit category $\Or(H)$. This gives an  exact sequence of the form
$$0 \to \RR_0 \to 2\uR{H/1} \to 3\uR{H/1} \to \uR{H/C_4}\oplus \uR{H/{C_2^B}}
 \oplus \uR{H/C_3} \to \HH_0 \to 0,$$ where all the modules
in the extension (excluding the ends) except $\uR{H/C_3}$ are free $\RG_H$-modules, and $\HH_0 = H_0(\uC{X})$. 

Since $R[H/C_3]$ is a projective $RH$-module (it is induced up from $R$,
which is projective over $R[C_3]$), we see that $\uR{H/C_3} = I_1R[H/C_3]$ as $\RG_H$-modules. Therefore $\uC{X}$ is a finite projective chain complex over $\RG_H$.

It is useful to represent an $\RG_H$-module $M$ by a labelled tree diagram:
$$M\quad =\qquad\vcenter {\xymatrix{M(C_4)\ar@{-}[d]&\\
M(C^A_2)\ar@{-}[d]& M(C^B_2)\ar@{-}[dl]\\
M(1)& }}$$ 
with one vertex for each isomorphism class of objects, and edges given by the partial ordering of the subgroups in $\cF$ up to conjugacy.
The labels are given by the $R[N_H(K)/K]$-modules $M(K)$, for $K \in \cF$.

For the homology module $H_0(X^{\textbf{?}};R)$ over
 the orbit category $\G_H$ we have the diagram
$$\HH_0 \quad =\qquad\vcenter{\xymatrix{R[D_8/C_4]\ar@{-}[d]&\\
R[D_8/C_4]\ar@{-}[d]&R[E/C_2^B]\ar@{-}[dl]\\
R& }}$$

The $(k+1)$-fold join of $\uC{X}$ with itself (see Section
\ref{sect:K-theory}) is a finite projective complex of the form
$$\CC : \quad 0\to \RR_0 \to C_n \to \dots \to C_k \to \dots \to C_0 \to \un{R} \to
0$$ over $\RG_H$ with $(n+1) = 3(k+1)$. 
In case  $X = S(V)$, where $V\cong R^3$ is a real orthogonal $H$-representation, then the join construction on spheres just produces the unit sphere $S(V\oplus \dots \oplus V)$ in the direct sum of $(k+1)$ copies of $V$. This sphere has real dimension $n=3(k+1) -1$. The purpose of the join construction is to produce a complex with dimension gaps between the non-zero homology groups, as required  by Theorem
\ref{thm:glueing} for glueing the different primes together.

We have  $H_0(\CC) = \un{R}$ and $H_n = \RR_0$.  If $(k+1)$ is even, then
$H_k (\CC(Q)) = R$, with trivial $N_H(Q)/Q$-action, and $\widetilde H_i (\CC(Q)) = 0$, for $i \neq k$,
for each non-trivial $Q \in \cF$. 
By Proposition \ref{prop:mislin_chain}, we obtain a chain
complex $\CC^{(2)}$ of projective $\RG _G$-modules, having homology isomorphic to $R$, with trivial $N_G(Q)/Q$-action. By construction, the homology dimension function $\un{n}$ for  $\CC^{(2)}$ is the same as for $\CC$. Notice that  $\un{n}$ is monotone, but not strictly monotone.

\subsection{The case $p=3$.}
Let $R = \bZ_{(3)}$ and $K = C_2^B$. The $3$-period of $G=S_5$ is
four \cite[Chap.~XII, Ex.~11]{cartan-eilenberg1},  so by Swan \cite{swan1}  there exists a periodic projective
resolution $\PP$ with
$$0 \to R \to P_n \to \dots \to P_1 \to P_0 \to R \to 0$$
over the group ring $RG$, for any  $n$ such that $4 \mid (n+1)$. We will assume that $12\mid (n+1)$, and  let $k$ be defined by the equation
$(n+1)=3(k+1)$. Similarly, since $N_G(K)/K \cong S_3$ also has
$3$-period $4$, we have a chain complex $\DD$ yielding a periodic
projective resolution $$0 \to R \to D_k \to \dots \to D_0 \to R \to
0$$ over $RS_3$. In the rest of this section we let  $W_K =
N_G(K)/K$ to simplify the notation.

We want a finite projective chain complex $\CC$
over $\RG_G$ which fits into an extension of chain complexes
$$ 0 \to E_1 \PP \to \CC \to I_{K } \DD \to 0$$
where the induced exact sequence on the $0$-th homology
$$ 0\to \RR_0\to H_0(\CC)\to I_{K}R\to 0$$
is the non-trivial extension ot tree diagrams (with vertices at $\{1, K\}$)
$$0\rightarrow \vcenter{\xymatrix{0\ar@{-}[d]\\ R}}
\rightarrow\vcenter{\xymatrix{R\ar@{-}[d]|{\vee}^{id}
\\ R}}\rightarrow\vcenter{\xymatrix{R\ar@{-}[d]\\ 0}}
\rightarrow 0.$$
For a projective $R[W_K ]$-module $D$, the module
$I_{K }D$  has a finite projective
resolution of the form
$$\xymatrix{0\ar[r]&E_1\Res_1E_{K}D\ar[r]&E_{K}D\ar[r]&I_
{K }D\ar[r]&0}.$$ 
By definition of the functors $E_x$ and $I_x$ (see Section \ref{sect:definitions}), the canonical map $$E_xM \to I_xM \to 0$$
 is always surjective 
for any $R[x]$-module $M$.
 We have $E_KR[W_K] = \uR{G/K}$ and
hence $E_{K}D$ is projective. Also $\Res _1 E_{K}D$ is projective,
because it is a summand of $R[G/K]$ which is projective as an
$\bZ_{(3)}G$-module. This shows that, once constructed, $\CC$ will
be homotopy equivalent to a finite projective chain complex by Lemma
\ref{lem:finiteness}.

Associated to every $RG$-chain map $f \colon \Res _1 E_{K} \DD \to
\PP$, there is a chain complex $\CC$ which fits into the push-out
diagram
$$\xymatrix{ 0\ar[r]&E_1\Res_1E_{K }\DD\ar[d]^{E_1 f}\ar[r]&E_{K }
\DD\ar[d]\ar[r]&I_{K }\DD\ar@{=}[d]\ar[r]&0\\
0\ar[r]&E_1\PP\ar[r]&\CC\ar[r]&I_{K }\DD\ar[r]&0 \ .}$$ We want to
choose $f$ so that $\CC$ satisfies the condition on homology. Note
that $$H_0 ( \Res _1 E_{K } \DD)=\Res _1 E_{K} R=R[G/ N_G (K)].$$
Since  the modules $\Res _1 E_{K} D_i =(E_K D_i) (1)$ are projective
for all $i$ and $\PP$ is exact, there exists a chain map $f\colon \Res _1
E_{K} \DD\to \PP$
$$\xymatrix{ \cdots
\ar[r]&(E_{K }D_1)(1)\ar[d]^{f_1}\ar[r]&(E_{K }D_0)(1)
\ar[d]^{f_0}\ar[r]&R[G/{N_G(K )}]\ar[d]^{\varepsilon}\ar[r]&0\\
\cdots\ar[r]&P_1\ar[r]&P_0\ar[r]&R\ar[r]&0 }$$ lifting the
augmentation map  $R[G/{N_G(K )}]\xrightarrow{\varepsilon}R$. To see
that the resulting push-out complex $\CC$  has the desired
properties, consider the homology at zero for the diagram of chain
complexes given above. Since $I_K$ is an exact functor, $H_1(I_{K
}\DD)=I_K H_1(\DD)=0$, and we get
$$\xymatrix{ 0\ar[r]&H_0(E_1\Res_1E_{K}\DD)\ar[d]^{H_0(E_1f)}\ar[r]
&H_0(E_{K }\DD) \ar[d]\ar[r]&H_0(I_{K }\DD)\ar@{=}[d]\ar[r]&0\\
0\ar[r]&H_0(E_1\PP)\ar[r]&H_0(\CC)\ar[r]&H_0(I_{K }\DD)\ar[r]&0 }
$$
where $H_0(E_1\PP) =E_1R$. Note that
$$H_0(E_1\Res_1E_{K }\DD) =
E_1\Res_1H_0(E_K\DD) = E_1\Res_1E_KR= E_1R[G/{N_G(K )}]. $$
This gives a diagram of the form
$$\xymatrix{&E_1\ker{\varepsilon}\ar[d]\ar@{=}[r]&E_1\ker{\varepsilon}
\ar[d]\\ 0\ar[r]&E_1R[G/{N_G(K
)}]\ar[d]^{E_1(\varepsilon)}\ar[r]&E_K R \ar[d]
\ar[r]&I_{K }R\ar@{=}[d]\ar[r]&0\\
0\ar[r]&E_1R\ar[r]&H_0(\CC)\ar[r]&I_{K }R\ar[r]&0 }
$$
where the middle vertical sequence of $\RG_G$-modules is given by
$$0\rightarrow \vcenter{\xymatrix{0\ar@{-}[d]\\{\ker
{\varepsilon}} }}
\rightarrow\vcenter{\xymatrix{R\ar@{-}[d]\\
R[G/{N_G(K )}]}} \rightarrow\vcenter{\xymatrix{R\ar@{-}
[d]|{\vee}^{id}\\
R}}\rightarrow 0$$ This shows that $H_0(\CC)$ has the desired
form.

Now, to obtain  the same homology dimension function as for the complex $\CC^{(2)}$, more homology must be added
to the complex $\CC$. We need to extend $\HH_k$ and $\HH_0$ via the
non-split extensions
$$ 0 \to \HH_0 \to \hat \HH_0 \to N \to 0 \quad \quad {\rm and}
\quad \quad 0 \to H_k \to \hat H_k \to N\to 0$$
where
$$N =\quad \vcenter{\xymatrix{R\ar@{-}[d]^{\id}&&\\
R\ar@{-}[d]&0\ar@{-}[dl]\\0&& }}$$ The module $N$ has a finite
projective resolution of the form
$$ 0 \to E_1 R[G/D_8 ] \to E_{C_4} R \to N \to 0.$$
Note that $\res_1 E_{C_4} R = R[G/N_G (C_4)]=R[G/D_8]$, and for
$Q=C_2 ^A$ we have
$$\res_{Q} E_{C_4} R=R \otimes _{R[D_8/C_4]} R[(G/C_4)^{Q}]
=R \otimes _{R[D_8/C_4]} R[N_G(Q)/N_{C_4}(Q)]=R$$ where the equality
in the middle comes from Lemma \ref{lem:fixedpointformula}. Since
$R$ is projective as an $R[D_8 /C_4]$-module, $E_{C_4}R$ is
projective. It is easy to see that $E_1 R[G/D_8 ]$ is also
projective. So, by Proposition \ref{prop:adding-subtracting}, we can
replace $\CC$ with a finite projective chain complex $\CC^{(3)}$ over $\RG_G$  which has the
desired homology.

\subsection{The case $p=5$.}
For $p=5$, the situation is easier than the case $p=3$.
Let $R = \bZ_{(5)}$. The $5$-period of $S_5$ equals $8$, so by
Swan \cite{swan1} there exists a periodic projective resolution $\PP$ over the group ring $RG$,  giving an exact sequence
$$0 \to R \to P_n \to \dots \to P_1 \to P_0 \to R \to 0$$
 for any positive integer $n$ such that $n+1=3
(k+1)$ for some integer $k$, with  $8 \mid (k+1)$. We start with the $\RG_G$-complex $\CC=E_1\PP$ obtained by the extension functor from $\PP$.
 Since $\CC$ has no homology at the non-trivial
$2$-subgroups in $\cF$, we need to change the homology at $\HH_0$ and
at $\HH_k$ to match the homology  we have for $p=2$ and $p=3$. Note that
we need to extend $H_k$ and $\HH_0$ via the non-split extensions
$$ 0 \to \HH_0 \to \hat \HH_0 \to M \to 0
\quad \quad {\rm and} \quad \quad 0 \to \HH_k \to \hat \HH_k \to M \to 0$$
where
$$M =\quad \vcenter{\xymatrix{R\ar@{-}[d]^{\id}&&\\
R\ar@{-}[d]&R\ar@{-}[dl]\\
0&& }}$$ Let $K=C_2 ^B$. The module $M$ is the direct sum of $L$
(which has the same form as $N$) and $I_{K}R$. We claim that each of
these modules have finite projective resolutions over $\RG_G$. For $I_K R$ we
have a resolution of the form
$$ 0 \to E_1 R[G/(K \times S_3)] \to E_K R \to
I _K R \to 0.$$ Note that $$\res _1 E_K R =R[G/N_G (K)]=R[G/(K
\times S_3)]$$ where $S_3$ denotes the subgroup of $S_5$ generated
by symmetries of $\{ 3,4,5 \}$. Since $R$ is projective as an
$R[N_G(K)/K]$-module, $E_K R$ is projective. It is clear that $E_1
R[G/(K\times S_3)]$ is also projective. So, the above resolution is
a projective resolution of $I_K R$. We can also write a finite
projective resolution for $L$ (similar to the resolution given for
$N$). So, by Proposition \ref{prop:adding-subtracting}, we can
replace $\CC$ with a finite projective chain complex $\CC^{(5)}$ which has the
desired homology.

\begin{proof}[The proof of Theorem A]
We will first construct a projective chain complex $\CC$
over $\ZG_G$ with isotropy in $\cF$,  by applying Theorem \ref{thm:glueing} to glue the $p$-local complexes $\CC^{(p)}$, for $p=2,3,5$. Note that
in the constructions of $\CC^{(p)}$ above,  we may choose any
 integer $k$  such that $k$ odd, $n+1 = 3(k+1)$, $12 \mid (n+1)$ and $8 \mid (k+1)$. To satisfy the first condition in Theorem \ref{thm:glueing}, that 
the distance  between non-zero homology groups of the $\CC^{(p)}$ is larger than $l(\G_G) = 2$, we will also need
  $k\geq 3$ and $n-k \geq 3$.  
  \begin{remark}
The minimum value for $k$ satisfying the requirements used above is $k=7$, which gives $n = \dim \CC= 23$.
\end{remark}
  The $\ZG_G$-module $\HH$ needed to satisfy the second condition in Theorem \ref{thm:glueing} is given by $\HH_i(K) = \bZ$, for $i = 0, \un{n}(K)$   with $K \in \cF$, and zero otherwise. 
 By Lemma \ref{lem:finiteness}, $\CC$ is chain homotopy equivalent to a finite projective complex.
  To obtain a finite free complex, we can
apply Theorem \ref{thm:K-theory}, which (possibly after some joins) produces a finite free $\ZG_G$-chain complex $\CC$
with the $\bZ$-homology of an $\un{n}$-sphere, and $\un{n}(K) \geqslant 3$ for all $K \in \cF$.

Note that our homology dimension function $\un{n}$ is not strictly monotone, since
$\un{n}(C_2^A) = \un{n}(C_4)$, but by Corollary \ref{cor:atmost2} we can modify
our complex to satisfy the conditions for geometric realization in Theorem \ref{thm:realization}, since $l(C_2^A, C_4) = 1$.
Applying Theorem \ref{thm:realization}, we conclude that $G=S_5$ acts on a finite
$G$-CW complex $X$ with isotropy  in $\cF$.
\end{proof}
\begin{remark} For this particular example we needed to apply Theorem
\ref{thm:K-theory} with one join tensor operation, because $\widetilde K_0(\ZG_G) = \bZ/2$. This follows from Theorem
\ref{thm:splittingK-theory}, Lemma \ref{lem:jointensorformula}  and  well-known calculations showing that
$\widetilde K_0(\bZ[N_G(Q)/Q]) = 0$, for $1\neq Q \in \cF$, but $\widetilde K_0(\bZ[G]) = \bZ/2$. Note that, by Dress induction, it is enough to consider the projective class groups of $p$-hyperelementary subgroups of $G$ (see \cite[\S 50]{curtis-reiner2}, \cite{reiner2}). 
We therefore obtain a finite $G$-CW complex $X\simeq S^{47}$ with isotropy in $\cF$.
\end{remark}

\providecommand{\bysame}{\leavevmode\hbox to3em{\hrulefill}\thinspace}
\providecommand{\MR}{\relax\ifhmode\unskip\space\fi MR }
\providecommand{\MRhref}[2]{%
  \href{http://www.ams.org/mathscinet-getitem?mr=#1}{#2}
}
\providecommand{\href}[2]{#2}

\end{document}